\documentclass[review]{elsarticle}

\usepackage{amsfonts, amssymb, amsmath, amsthm}
\usepackage[colorlinks=true,linkcolor=red]{hyperref}
\usepackage{graphicx}
\usepackage{color}

\theoremstyle{plain}
\newtheorem{thm}{Theorem}

\newtheorem{lem}{Lemma}
\newtheorem{prp}{Proposition}

\theoremstyle{definition}
\newtheorem{dfn}{Definition}

\newcommand{\mrank}{\mathrm{rank}}

\usepackage{hyperref}
\journal{Journal of \LaTeX\ Templates}

\newcommand{\mcM}{\mathcal{M}}
\newcommand{\mcR}{\mathcal{R}}

\newcommand{\mbk}{\mathbf{k}}
\newcommand{\mbq}{\mathbf{q}}
\newcommand{\mbn}{\mathbf{n}}
\newcommand{\mbl}{\mathbf{l}}

\newcommand{\mbZ}{\mathbb{Z}}
\newcommand{\mbR}{\mathbb{R}}

\begin{document}

\begin{frontmatter}

\title{Locally optimal 2-periodic sphere packings}

\author{Alexei Andreanov\corref{mycorrespondingauthor}}
\cortext[mycorrespondingauthor]{Corresponding author}
\address{Center for Theoretical Physics of Complex Systems, Institute for Basic Science (IBS), Daejeon 34126, Republic of Korea}
%\address{Center for Theoretical Physics of Complex Systems, Institute for Basic Science, Daejeon, South Korea}
\ead{alexei348@ibs.re.kr}

\author{Yoav Kallus}
\address{Santa Fe Institute, 1399 Hyde Park Road, Santa Fe, NM 87501, USA}
\ead{Yoav.Kallus@gmail.com}

\begin{abstract}
    The sphere packing problem is an old puzzle. We consider packings with $m$ spheres in the unit cell ($m$-periodic packings). For the case $m=1$ (lattice packings), Voronoi proved there are finitely many inequivalent local optima and presented an algorithm to enumerate them, and this computation has been implemented in up to $d=8$ dimensions. We generalize Voronoi's method to $m>1$ and present a procedure to enumerate all locally optimal 2-periodic sphere packings in any dimension, provided there are finitely many. We implement this computation in $d=3,4$, and $5$ and show that no 2-periodic packing surpasses the density of the optimal lattices in these dimensions. A partial enumeration is performed in $d=6$.
\end{abstract}

\begin{keyword}
    sphere packing\sep periodic point set\sep quadratic form\sep Ryshkov polyhedron
    \MSC[2010] 	52C17\sep 11H55\sep 52B55\sep 90C26
    %11H55 = geometry of numbers, quadratic forms
    %52B55 = Computational aspects related to convexity
    %52C17 = Packing and covering in $n$ dimensions
    %90C30  	Nonlinear programming
    %90C26  	Nonconvex programming, global optimization

\end{keyword}

\end{frontmatter}

%\linenumbers

\section{Introduction}
\label{sec:intro}

The sphere packing problem asks for the highest possible density achieved by an arrangement of nonoverlapping spheres in a Euclidean space of $d$ dimensions. The exact solutions are known in $d=2$~\cite{toth1940uber}, $d=3$~\cite{hales2005proof}, $d=8$~\cite{viazovska2016sphere} and $d=24$~\cite{cohn2016sphere}. This natural geometric problem is useful as a model of material systems and their phase transitions, and, even in nonphysical dimensions,it is related to fundamental questions about crystallization and the glass transition~\cite{charbonneau2014fractal}.The sphere packing problem also arises in the problem of designing an optimal error-correcting code for a continuous noisy communication channel~\cite{conway1999sphere}. Perhaps even more than the obvious applications of the sphere packing problem, contributing to its importance is the unexpected wealth of remarkable structures that appear as possible solutions and merit study in their own right~\cite{thompson1983error,cohn2017conceptual}.

There does not appear to be a systematic solution or construction that achieves the optimal packing in every dimension, and every dimension seems to have its own quirks, unearthing new surprises~\cite{conway1999sphere}. In some dimensions, such as $d=8$ and $d=24$, the solution is unique and given by exceptionally symmetric lattices. In others, such as $d=3$, there is a lattice that achieves the highest density, but this density can also be achieved by other periodic packings with larger fundamental unit cells and even by aperiodic packings. For $d=10$ it seems that the highest density, achieved by a periodic arrangement with $40$ spheres per unit cell, cannot be achieved by a lattice. In any given dimension, the densest packing known is periodic, but it is not known whether there is some dimension where the densest packing is not periodic. In fact, frighteningly little is known in general as we go up in dimensions. It is possible that in all dimensions the densest packing is achieved by a periodic packing with a universally bounded number of spheres in the unit cell, that in all dimensions it is achieved by a periodic packing, but only with unboundedly many spheres in the unit cell, or that in some dimension it is not achieved by a periodic packing at all. In any case, periodic packings with arbitrarily many spheres in the unit cell can approximate the optimal density in any dimension to arbitrary precision.

Therefore, it is reasonable to ask for the optimal density achieved by a periodic packing in $d$ dimensions with $m$ spheres per unit cell, which we denote $\phi_{d,m}$. In the limit $m\to\infty$, this density approaches (and possibly equals for some finite $m$) the optimal packing density $\phi_d$. Apart from the dimensions where $\phi_d$ is known, there are no cases with $m>1$ where $\phi_{d,m}$ is known. In this paper, we describe a method for calculating $\phi_{d,2}$, and use it to obtain $\phi_{4,2}=\phi_{4,1}$ and $\phi_{5,2}=\phi_{5,1}$. Our method enumerates all inequivalent local optima, and therefore does not terminate if there are infinitely many. However, we conjecture that there are finitely many inequivalent local optima in each dimension, and therefore that our method reduces the problem of calculating $\phi_{d,2}$ to a finite calculation. Our method could also be generalized to larger $m$, but becomes more complicated.

For the case $m=1$, corresponding to lattices, Voronoi gave an algorithm to enumerate all the locally optimal solutions to the problem. In geometrical terms, Voronoi's algorithm uses the fact that the space of lattices up to isometry can be parameterized (redundantly) by positive definite quadratic forms. The subset in the linear space of quadratic forms that corresponds to lattices with no two points less than a certain distance apart is a locally finite polyhedron (a closed convex set, every intersection of which with a compact polytope is a compact polytope; see \ref{sec:lfp}). The additional fact that the density of lattice points is a quasiconvex function implies that local maxima can only occur at the vertices of this polyhedron. Because there are only finitely many vertices that correspond to distinct lattices, the local optima of the lattice sphere packing problem can be fully enumerated. This interpretation of Voronoi's algorithm was suggested by Ryshkov~\cite{ryshkov1970polyhedron} and so the polyhedron is known as the Ryshkov polyhedron. The lattices corresponding to vertices of the polyhedron are known as \textit{perfect lattices}, but not every perfect lattice is locally optimal, or an \textit{extreme lattice}. This algorithm has been implemented and executed to determine $\phi_{d,1}$ for $d\le 8$, as well as to enumerate all perfect lattices in these dimensions~\cite{schurmann2009computational}. However, the exploding number of perfect lattices as $d$ increases make the execution of this algorithm currently unfeasible even for $d=9$~\cite{andreanov2012random}.

Sch\"{u}rmann transported some of the concepts from Voronoi's theory to the case of periodic packings and was able to show that any extreme lattice, when viewed as a periodic packing, is still locally  optimal~\cite{schurmann2010perfect,schurmann2013strict}. However the nonlinearities that arise in the case $m>1$ have made it difficult to use this version of the theory to provide an analog of Voronoi's algorithm. In this paper we transport the concepts in a slightly different way than Sch\"{u}rmann, permitting us to reconstruct many of the elements of Voronoi's algorithm in the case of periodic packings with fixed $m$. These elements behave nicely enough for $m=2$ to allow us to provide an algorithm for the enumeration of all locally optimal packings, and we execute this algorithm for $d\le5$.

\section{Theoretical preparation}
\label{sec:tp}

\subsection{Periodic point sets}

For a set of points, $\Xi\subset\mbR^d$, the \textit{packing radius} is the largest radius, such that if balls were centered at each of the points, they would not overlap. Namely, the packing radius is half the infimum distance between any two points, $\rho(\Xi) = \tfrac12\mathrm{inf}_{x,y\in\Xi}\|x-y\|$. A set of points is \textit{periodic} if there are $d$ linearly independent vectors $\mathbf{a}_1,\ldots,\mathbf{a}_d$, such that the set is invariant under translation by any of these vectors. The group generated by $\mathbf{a}_1,\ldots,\mathbf{a}_d$ is a lattice, $\Lambda =\{\sum_{i=1}^{d} n_i \mathbf{a}_i : n_1,\ldots,n_d\in \mbZ\}$. If all translations that fix the set are in $\Lambda$, we say that $\Lambda$ is a primitive lattice for the periodic set. The primitive lattice is unique. A set that is periodic under translation by a lattice $\Lambda$ and has $m$ orbits under translations by $\Lambda$ called $m$-periodic. We will only deal with sets of positive packing radius, and therefore $m$ is necessarily finite.

The number density of a periodic set $\Xi$ is given by
\begin{equation}
    \delta(\Xi) = \frac{m}{\det \Lambda}\text,
\end{equation}
where $\det \Lambda$ is the volume of a parallelotope generated by the generators of $\Lambda$, $\{\sum_{i=1}^{d} x_i \mathbf{a}_i : 0\le x_1,\ldots,x_d\le 1\}$. This parallelotope and its translates under $\Lambda$ tile $\mbR^d$, and its volume is the same as the volume of any other \textit{fundamental cell} of $\Lambda$, that is, a polytope whose $\Lambda$-translates tile $\mbR^d$. Therefore, this is a property of $\Lambda$, and not of the particular basis. Also, since $\Lambda$ itself is not uniquely determined from $\Xi$, but can be any sublattice of the primitive lattice of $\Xi$, it is important to note that the formula for $\delta(\Xi)$ above is independent of the choice of $\Lambda$. The largest density that can be achieved by a packing of equal-sized balls centered at the points of $\Xi$ is $\phi(\Xi) = V_d \rho(\Xi)^d \delta(\Xi)$, where $V_d$ is the volume of a unit ball in $\mbR^d$. In any family of point sets closed under homothety, such as periodic sets or $m$-periodic sets, maximizing the packing density $\phi$ is equivalent to maximizing the number density under the constraint $\rho(\Xi)\ge\rho_0$ for some fixed $\rho_0$.

\subsection{Voronoi's algorithm}

A $1$-periodic packing is a translate of a lattice. The packing radius of a lattice is given by
\begin{equation}
    \rho(\Lambda)=\tfrac12\min_{\mbl,\mbl'\in\Lambda}\|\mbl-\mbl'\|=\tfrac12\min_{\mbl\in\Lambda}\|\mbl\|\text.
\end{equation}
The packing radius of a lattice and the volume of its fundamental cell are both invariant under rotations of the lattice. Therefore, to find the densest lattice packing, we only need to consider lattices up to rotation. Consider a lattice $\Lambda=A\mbZ^d$ generated by $\mathbf{a}_1,\ldots,\mathbf{a}_d$. The quadratic form $Q\colon \mbZ^d\to\mbR$, $Q(n_1,\ldots,n_d) = \|n_1 \mathbf{a}_1+\ldots+n_d \mathbf{a}_d\|^2$ determines $\Lambda$ up to rotations. However, since a lattice can have different generating vectors, different quadratic forms can correspond to the same lattice. $Q$ and $Q'$ correspond to the same lattice if and only if $Q=Q\circ U$, where $U\in GL_d(\mbZ)$. This is precisely the group of linear maps that map $\mbZ^d$ to itself. The packing radius of a lattice $\Lambda$ corresponding to the quadratic form $Q$ is $\rho(\lambda) = \tfrac 12 (\min Q)^{1/2}$, where $\min Q$ is the minimum of $Q$ over nonzero vectors. The determinant of $\Lambda$ is given by $(\det Q)^{1/2}$.

The linear space of quadratic forms can be identified with the linear space of symmetric matrices $\mathcal{S}^d$ using the standard basis of $\mbZ^d$, so that $Q(\mbn) = \mbn^T Q\mbn$. And the natural inner product in this space is given by $\langle Q, Q'\rangle = \operatorname{tr} QQ'$. The condition $\min Q \ge \lambda$ can be written as the intersection of infinitely many linear inequalities:
\begin{equation}
    \label{eq:ineqlatt}
    Q(\mbn) = \langle \mbn\mbn^T, Q \rangle \ge \lambda \quad \text{ for all } \mbn\in\mbZ^d
\end{equation}
It can be easily checked that for $\lambda>0$, Eq.~\eqref{eq:ineqlatt} implies that $Q$ is a positive definite matrix,
and a lattice corresponding to it can be recovered, e.g.\, by Cholesky decomposition.

The set $\{Q: \min Q \ge \lambda\}$ is a locally finite polyhedron, that is, a closed convex set such that every intersection of it with a compact polytope is a compact polytope. Loosely speaking, though it is not a polyhedron (the intersection of finitely many halfspaces), it behaves locally like one. We were not able to find a satisfactory reference on locally finite polyhedra, and so we prove some basic results about them in \ref{sec:lfp}. In the remainder of the text, we use the less cumbersome \textit{polyhedron} to refer to any locally finite polyhedron. Ryshkov reinterpreted Voronoi's algorithm in terms of the polyhedron $\{Q: \min Q \ge \lambda\}$, which is therefore known as the Ryshkov polyhedron~\cite{ryshkov1970polyhedron}. Though the Ryshkov polyhedron has infinitely many vertices, there are only finitely many orbits under the action of $GL_d(\mbZ)$. Therefore, if we start at some vertex, enumerate the vertices with which it shares an edge, and repeat the process for any enumerated vertex not of the same orbit as a previous vertex, we eventually have representatives of all the orbits of vertices in a connected component of the 1-skeleton of the Ryshkov polyhedron. Since the 1-skeleton of a locally finite polyhedron is connected, these are all the orbits of vertices.

We wish to find the minimum of $\det Q$ in this polyhedron. Because the determinant is a quasiconcave function, its local minima in the polyhedron can occur only on the vertices, and to find the global minimum it is enough to compare its value at all the vertices.

\subsection{The Ryshkov-like Polyhedron}

To parameterize $m$-periodic point sets, we can use a similar scheme. An $m$-periodic point set is the union of $m$ translates of a lattice, $\Xi = \bigcup_{i=0}^{m-1} (\Lambda + \mathbf{b}_i)$, where $\Lambda = A\mbZ^d$. Without loss of generality, we can take $\mathbf{b}_0 = 0$. Now, consider the set $\mathcal M\subseteq\mbZ^{d+m-1}$ given by $\mathcal M = \mbZ^d\times(E-E)$, where $E=\{0,\mathbf{e}_1,\ldots,\mathbf{e}_{m-1}\}$ is the standard basis of $\mbZ^{m-1}$ plus the zero vector. We define the following function $J:\mathcal M\to\mbR$:
\begin{equation}
    J(n_1,\ldots,n_d,l_1,\ldots,l_{m-1}) = \|n_1 \mathbf{a}_1+\ldots+n_d \mathbf{a}_d+l_1 \mathbf{b}_1+\ldots+l_{m-1} \mathbf{b}_{m-1}\|^2\text.
\end{equation}
This function can be extended uniquely to a quadratic form over $\mbR^{d+m-1}$, represented by a symmetric matrix $J\in\mathcal{S}^{d+m-1}$, namely
\begin{equation}
    \label{eq:Jblocks}
    J = \left(\begin{array}{cc} Q& R^T\\ R& S\end{array}\right)\text,
\end{equation}
where $Q_{ij} = \mathbf{a}_i\cdot\mathbf{a}_j$, $R_{ij}=\mathbf{b}_i\cdot\mathbf{a}_j$, and $S_{ij}=\mathbf{b}_i\cdot\mathbf{b}_j$. The quadratic form, and therefore also the function $J: \mathcal M\to\mathbb R$, determines $\Xi$ up to rotation.

The packing radius of $\Xi$ is given by
\begin{equation}
    \begin{aligned}
	\rho(\Xi) &= \min_{\mathbf{l}, \mathbf{l}'\in \Lambda, 0\le i,j\le m-1} \|\mathbf{l}-\mathbf{l}'+\mathbf{b}_i-\mathbf{b}_j\|\\
	&=\left(\min_{\mathbf{l}, \in \Lambda, 0\le i,j\le m-1} \|\mathbf{l}+\mathbf{b}_i-\mathbf{b}_j\|^2\right)^{1/2}\\
	&=\tfrac12 (\min J)^{1/2}\text,
    \end{aligned}
\end{equation}
where $\min J$ is the minimum of $J$ over $\mathcal M\setminus\{0\}$.

Let $\mcR(\lambda) = \{ J:\min J\ge\lambda\}$. This set is defined as the intersection of linear inequalities:
\begin{equation}
    \label{eq:ineqper}
    J\in \mcR(\lambda) \quad \text{ iff } \quad J(\mbk) = \langle \mbk\mbk^T, J \rangle \ge \lambda \quad \text{ for all } \mbk\in\mcM\setminus\{0\}\text.
\end{equation}
We want to show that $\mcR(\lambda)$, like the Ryshkov polyhedron, is a locally finite polyhedron.

\begin{lem}
    Let $Q:\mbR^d\to\mbR$ be a positive definite quadratic form, satisfying $Q(\mbn)\ge\lambda$ for all $\mbn\in\mbZ^d$ and
    $\mathrm{tr}\,Q\le C$. Then any vector $\mathbf{x}\in\mbR^d$ satisfying $Q(\mathbf{x})\le1$ also satisfies
    $\|\mathbf{x}\|\le M$, where $M$ depends on $d$, $\lambda$, and $C$.
\end{lem}
The proof is given in Theorem 3.1 of~\cite{schurmann2009computational}.

\begin{thm} 
    Let $W=\{J\in \mcR(\lambda):|J_{ij}|\le L\}$.
    There are only finitely many $\mbk\in\mathcal M$ such that for some $J\in W$, $J(\mbk)\le\lambda$ .
\end{thm}

\begin{proof}
Let $\mbk=(\mbn,\mbl)$. There are only finitely many choices for $\mbl$. So, we need only show that for any fixed choice of $\mbl$, there are finitely many $\mbn\in\mbZ^d$ such that $J(\mbk)\le\lambda$ for some $J\in W$. We have 
\begin{equation}
    \begin{aligned}
	\label{eq:inhom}
	J(\mbk) &= \mbn^TQ\mbn + 2\mbl^T R\mbn+\mbl^T S\mbl\\
	&= (\mbn+\mbq)^TQ(\mbn+\mbq)+\mbl^T S \mbl-\mbq^T Q\mbq\text,
    \end{aligned}
\end{equation}
where $\mbq=Q^{-1}R^T\mbl$. Therefore, if $J(\mbk)\le\lambda$ then $(\mbn+\mbq)^TQ(\mbn+\mbq)\le\lambda+\mbq^TQ\mbq-\mbl^TS\mbl\le\lambda+\mbq^TQ\mbq = \lambda+\mbl^TRQ^{-1}R^T\mbl$.

Since $\mbl^TRR^T\mbl\le 4d L^2$ and $\mathrm{tr}Q\le dL$, we have from the lemma that $\mbl^TRQ^{-1}R^T\mbl\le 4dL^2 M^2$. Therefore, if $J(\mbk)\le\lambda$ then $(\mbn+\mbq)^TQ(\mbn+\mbq)\le \lambda+4dL^2 M^2$. Again, from the lemma, we have that $(\mbn+\mbq)^T(\mbn+\mbq)\le \lambda M^2 + 4dL^2 M^4$. Therefore, there can only be finitely many choices of $\mbn$ for each choice of $\mbl$ such that $J(\mbk)\le\lambda$ for some $J\in W$.
\end{proof}

So, any intersection of $\mcR(\lambda)$ with a compact polytope is a compact polytope and $\mcR(\lambda)$ is a locally finite polyhedron. We call it the Ryshkov-like polyhedron. Moreover, since each vertex of $\mcR(\lambda)$ is the unique solution of a linear system of equations with integer coefficients, the vertices of $\mcR(\lambda)$ have rational coordinates.

\subsection{Symmetries of the Ryshkov-like polyhedron}
\label{sec:symmetries}

The symmetries of $\mcR(\lambda)$ are tightly linked with the symmetries of the set $\mathcal M\subset\mbZ^{d+m-1}$. Namely, if $T:\mbR^{d+m-1}\to\mbR^{d+m-1}$ is a linear map such that $T(\mathcal M) = \mathcal M$, then $J\mapsto J\circ T$ maps $\mcR(\lambda)$ to itself.

We decompose $T$ into blocks with a top-left block of size $d\times d$. It is easy to see that the bottom-left block must be zero, or else there is always some $\mbk\in\mathcal M$ such that $T(\mbk) = (\mbn,\mbl)$, where $\mbl\not\in (E-E)$. Therefore we write
\begin{equation}
    \label{eq:Tblocks}
    T = \left(\begin{array}{cc} U& V\\ 0& W\end{array}\right)\text.
\end{equation}
As a consequence, we also have that $U$ and $W$ must be invertible as maps $\mbZ^d\to\mbZ^d$ and $(E-E)\to(E-E)$. Therefore, $U\in GL_d(\mbZ)$ and $W$ is a permutation of $E$, namely
\begin{equation}
    W = \left(\begin{array}{cccc} 0 & 1 & 0 & \cdots \\ 0 & 0 & 1 & \cdots \\ \vdots & \vdots & \vdots &\ddots\end{array}\right)
    \Pi \left(\begin{array}{ccc} -1 & -1 & \cdots \\ 1 & 0 & \cdots \\ 0 & 1 & \cdots \\ \vdots & \vdots & \ddots\end{array}\right)\text,
\end{equation}
where $\Pi$ is a $m\times m$ permutation matrix.

It is easy to check that whenever $U\in GL_n(d)$, $W$ is a permutation of $E$, and $V$ is an arbitrary $d\times(m-1)$ integer matrix, then $T$ is a symmetry of $\mathcal M$. Let us call this group of symmetries $\Gamma$, and use its members to act on elements of $\mathcal M$ by $\mbk\mapsto T\mbk$ or elements of $\mcR(\lambda)$ by $J\mapsto J\circ T$. We conjecture that for any $\lambda>0$, $\mcR(\lambda)$ has finitely many faces (in particular vertices) up to the action of $\Gamma$, but we do not have a proof, except in the cases where we have enumerated the vertices (see Sec.~\ref{sec:nr}), and found, by the fact of the algorithm halting, that there were finitely many vertices.

\subsection{The rank constraint}
\label{sec:rankconstraint}

Equation~\eqref{eq:Jblocks} described how to obtain a quadratic form $J\in\mcR(\lambda)$ from an $m$-periodic packing $\Xi$ of packing radius $\rho(\Xi) \ge \tfrac12 \lambda^{1/2}$. However, the reverse operation is not always possible. Clearly, $\mrank\,J = d$ is a necessary condition. In fact, it is also sufficient, since $J\in\mcR(\lambda)$ for $\lambda>0$ implies that $Q$ is positive definite, and therefore $\mathbf{a}_1,\ldots,\mathbf{a}_d$ can be recovered through, e.g., Cholesky decomposition, and $\mathbf{b}_1,\ldots,\mathbf{b}_{m-1}$ by solving $R_{ij}=\mathbf{b}_i\cdot\mathbf{a}_j$.

Let $\mcR_0(\lambda) = \{J\in\mcR(\lambda): \mrank\,J = d\}$. Since $Q$ is positive definite within $\mcR(\lambda)$, we can write this condition as the vanishing of the Schur complement $S - R Q^{-1} R^T = 0$. For $m=2$, we also have two simpler equivalent forms for this constraint: $\det J = 0$ or $\lambda_\text{min}(J) = 0$, where $\lambda_\text{min}$ denotes the smallest eigenvalue.

\subsection{The density objective}\label{sec:density}

We wish to find the maximum of $\delta(\Xi)$ among $m$-periodic sets $\Xi$ of packing radius at least $\rho_0$. This is equivalent to the minimization problem
\begin{equation}
    \label{eq:vol-J}
    \text{minimize } \mathrm{obj}(J) = \det Q(J)
    \text{ subject to } J\in \mcR_0(4\rho_0^2)\text,
\end{equation}
where $Q(J)$ is the top-left block of $J$ (see \eqref{eq:Jblocks}). Since the problem is equivalent for different values of $\rho_0$, we will often use for convenience the explicit value $\rho_0=1$, which corresponds to $\lambda = 4\rho_0^2 = 4$. While $\mathrm{obj}(J)$, like the objective in Voronoi's algorithm, is quasiconcave, the nonlinearity of the rank constraints does not allow for a straightforward characterization of the local minima.

A vivid illustration of a new type of local minima that can arise is a family of 9-dimensional 2-periodic packing arrangements known collectively as the \textit{fluid diamond packing}~\cite[p.\ 41]{schurmann2009computational}. The lattice $D_9 = \{\mbn\in\mbZ^n: n_1+n_2+\ldots+n_9 \in 2\mbZ\}$ has a packing radius $\rho = 1/\sqrt{2}$ and is an extreme lattice,  so that any nearby lattice has a smaller packing radius or a larger determinant. The deep holes of a lattice are the points that maximize the distance to the nearest lattice point. For $D_9$, the deep holes are at $D_9 + \mathbf{b}_0$, where $\mathbf{b}_0 = (\tfrac12,\tfrac12,\ldots,\tfrac12)$, and are at a distance of $3/2>2\rho$ from the lattice. So, $D_9^+(\mathbf{b}) = D_9 \cup (D_9+\mathbf{b})$ has the same packing radius, but double the density when $\mathbf{b}$ is sufficiently close to $\mathbf{b}_0$. As a 2-periodic point set, this is clearly a local maximum of the density since a nearby 2-periodic point must be of the form $\Lambda\cup(\Lambda+\mathbf{b}')$, where $\Lambda$ is nearby $D_9$, and therefore has a larger determinant or a packing radius that is no larger. Thus we have a 9-dimensional family 2-periodic point sets that are all locally optimal. In fact, the fluid diamond packing achieves the highest known packing density for spheres in 9 dimensions. A similar situation can be constructed with any extreme lattice whose covering radius (the distance of the deep holes from the lattice) is more than twice its packing radius. Another example, therefore can be constructed with the lattice $A_8$, but this example is less often discussed, presumably because other packings achieve higher density in that dimension.

This vivid example should discourage us from attempting to transport the methods of the Voronoi algorithm, where a critical result was that all local optima lie on vertices and are therefore isolated. However, we will show that, at least in the case $m=2$, the packing constructed this way are the only possible problem we may encounter. Namely, for $m=2$, a local minimum of $\mathrm{obj}(J)$ in $\mcR_0(\lambda)$ is either an isolated minimum lying on an edge of $\mcR(\lambda)$ or is the union of two translates of an extreme lattice whose covering radius is more than twice the packing radius.

To start, we define a sufficient condition for local optimality by linearizing all of the constraints, which we call \textit{algebraic extremeness}, and prove that it is indeed a sufficient condition. Let $J\in\mcR_0(\lambda)$. The set $\{J':\mrank\, J' = d\}$ is a $[\tfrac12(d+m)(d+m-1) - \tfrac12 m(m-1)]$-dimensional manifold in the neighborhood of $J$, and its tangent space is given by $$T_J = J + \left(\mathcal S[N(J)]\right)^\perp = \{ J' : \langle J' - J, \mathbf{u}\mathbf{u}^T\rangle = 0 \text{ for all } \mathbf{u}\in N(J)\}\text,$$ where $N(J)$ is the null space of $J$, and $\mathcal S[N(J)]$ is the space of quadratic forms over $N(J)$. The linear inequality constraints, encoded in the polyhedron $\mcR(\lambda)$, give rise in the neighborhood of $J$ to the cone $$C_J = \{ J' : J'(\mbk) \ge \lambda\text{ for all } \mbk\in\mcM\text{ such that }J(\mbk)=\lambda\}\text.$$ Finally, the gradient of the objective $\mathrm{obj}(J)$ is proportional to $$G_J = \left(\begin{array}{cc} Q^{-1} & 0\\0 & 0\end{array}\right)\text.$$

\begin{dfn}
    \label{def:algext}
    A configuration $J\in\mcR_0(\lambda)$ is called \textit{algebraically extreme}
    if $\langle J'-J,G_J\rangle > 0$ for all $J'\in C_J\cap T_J$ except $J'=J$.
\end{dfn}

\begin{thm}
    \label{thm:suff}
    Let $J\in\mcR_0(\lambda)$ be algebraically extreme. Then
    \begin{enumerate}
       \item $J$ is an isolated minimum of $\mathrm{obj}(J')$~\eqref{eq:vol-J} among $J'\in\mcR_0(\lambda)$.
       \item $J$ lies on a $\tfrac12 m(m+1)$-dimensional face of $\mcR(\lambda)$.
   \end{enumerate}
\end{thm}
\begin{proof}
    The first claim follows as the usual sufficient optimality criterion for a smooth optimization problem with smooth equality and inequality constraints \cite{bazaraa2006nonlinear}.

    To prove the second claim, note that if $J$ does not lie on a $\tfrac12 m(m+1)$-dimensional face of $\mcR(\lambda)$, then
    $C_J$ contains an affine space passing through $J$ of dimension $\tfrac12 m(m+1)+1$. Since $T_J$ is an affine space of codimension $\tfrac12m(m+1)$
    passing through $J$, then $C_J\cap T_J$ must include a line, and it is not possible
    for $\langle J'-J,G_J\rangle$ to be positive everywhere on this line, even with $J$ excepted. Therefore, $J$ is not algebraically extreme.
\end{proof}

\begin{dfn}
    \label{def:fluid}
    A configuration $J\in\mcR_0(\lambda)$ is called a \textit{fluid packing}
    if it sits on a nonconstant continuous curve $J'\colon[0,1]\to\mcR_0(\lambda)$ of the form
    \begin{equation}
	\label{eq:Jblocks-fluid}
	J'(t) = \left(\begin{array}{cc} Q& R+t \tilde{R}\\ R^T + \tilde{R}^T& S+ t\tilde{S}+t^2\tilde{R}^T Q^{-1} \tilde{R}
	\end{array}\right)\text,
    \end{equation}
    such that $J'(0)=J$ and $J'(t)$ is a local minimum
    of $\mathrm{obj}$ over $\mcR_0(\lambda)$ for all $0\le t \le 1$.
\end{dfn}

One example of a fluid packing is $D_9^+(\mathbf{b}_0)$, which is part of the 9-dimensional fluid diamond packing described above. Specifically, let $A$ be a matrix whose rows form a basis for $D_9$, let $\mathbf{b}_1$ be a point such that $\|\mathbf{b}_0-\mathbf{b}_1\| < 3/2 -\sqrt{2}$, and let
\begin{equation}
    \label{eq:Jblocks-fluid-diamond}
    J'(t) = \left(\begin{array}{cc} AA^T& (1-t) A\mathbf{b}(t)\\ \mathbf{b}(t)^TA^T& \|\mathbf{b}(t)\|^2\end{array}\right)\text,
\end{equation}
where $\mathbf{b}(t) = (1-t) \mathbf{b}_0 + t \mathbf{b}_1$.

When $m=2$, every fluid packing is the union of two extreme lattices. The continuous curve in Definition \ref{def:fluid} when $m=2$ must be of the form \eqref{eq:Jblocks-fluid-diamond}. Therefore, $J'(t)(\mathbf k)$ is a nonconstant quadratic function of $t$ for every $\mathbf{k}=(\mathbf{n},\pm 1)$ and there is some $0 < t' < 1$ such that $J'(t)(\mathbf{k}) > \lambda$ for every such $\mathbf{k}$. Because $J'(t')$ is a local minimum of $\mathrm{obj}$ over $\mcR_0(\lambda)$, $Q$ is a local minimum of $\det Q''$ over $\{ Q'': \mathbf{n}^T Q''\mathbf{n} \ge \lambda \text{ for all } \mathbf{n}\in\mathbf{Z}^d\}$.

\begin{thm}
    \label{thm:deg}
    Let $m=2$, and let $J\in\mcR_0(\lambda)$. If $J$ is a local minimum of $\mathrm{obj}(J')$ among
    $J'\in\mcR_0(\lambda)$ then one of the following is true:
    \begin{enumerate}
	\item $J$ is algebraically extreme.
	\item $J$ is a fluid packing.
    \end{enumerate}
\end{thm}
\begin{proof}
    Without loss of generality, let $\lambda=4$. Let $\{\mbk_1,\ldots,\mbk_s\} = \{\mbk:J(\mbk)=4\} = J^{-1}(4)$.
    The quadratic form $J$ is a local optimum of the problem
    \begin{equation}
        \label{eq:program}
	    \begin{aligned}
	        \text{Minimize\ } & \mathrm{obj}(J') \text,\\
	        \text{subj.\ to\ } & g_i(J') \triangleq 4-\langle\mbk_i\mbk_i^T,J'\rangle \le 0 \text{ for } i=1,\ldots,s \text,\\
	        & h(J') \triangleq \lambda_\text{min} (J') = 0\text.
	    \end{aligned}
    \end{equation}
    As a consequence, it satisfies the Karush-Kuhn-Tucker condition. Namely, there exist $u_\mbk\ge 0$ and $v$ such that
    $\nabla \mathrm{obj} + \sum_{i=1}^{s} u_i\nabla g_i + v \nabla h = 0$.
    We consider the derivative with respect to $J'_{d+1,d+1}$, which we denote $\partial_{d+1,d+1}$.
    Since $\partial_{d+1,d+1} \mathrm{obj} = 0$, $\partial_{d+1,d+1} g_i \le 0$, and $\partial_{d+1,d+1} h > 0$, we have that $v\ge0$.
    The Lagrangian corresponding to these KKT coefficients is $L = \mathrm{obj} + \sum u_i g_i + v h$.
    Let $C_0$ be the marginal cone
    \begin{equation}
	    \begin{aligned}
	        C_0=\{\tilde{J}: &\langle \nabla g_i,\tilde{J}\rangle = 0 \text{ for all }i\in I^+,\\
	        & \langle \nabla g_i,\tilde{J}\rangle \ge 0 \text{ for all }i\in I^0,\\
	        & \langle\nabla h,\tilde{J}\rangle=0\}\text,
	    \end{aligned}
    \end{equation}
    where $I^+ = \{i : u_i>0\}$ and $I^0=\{i:u_i=0\}$.
    Suppose that $J$ is not algebraically extreme, then the marginal cone is nontrivial. Consider the
    Hessian of the Lagrangian $H = \mathrm{hess}\,L = \mathrm{hess}\,\mathrm{obj} + v \mathrm{hess}\,h$, where, e.g.,
    $\langle \tilde J, (\mathrm{hess}\,\mathrm{obj})\tilde J\rangle = (d^2/dt^2)\mathrm{obj}(J+t\tilde{J})$.
    A necessary condition for $J$ to be a local optimum of \eqref{eq:program} is that $\langle \tilde{J},H\tilde{J}\rangle$
    is nonnegative for all $\tilde{J}\in C_0$ \cite[p.\ 216]{bazaraa2006nonlinear}.
    However, both $\mathrm{hess}\,\mathrm{obj}$ and $\mathrm{hess}\,h$ are negative semidefinite (for $\mathrm{obj}$, this
    is the quasiconcavity of the determinant function; for $h$, this is a well-known result of second-order eigenvalue
    perturbation theory). It follows that $C_0$ must lie in the null-space of $\mathrm{hess}\,\mathrm{obj}$, that is, the $Q$ component
    of any vector in the marginal cone is zero.

    Consider a nonzero element $\tilde{J}\in C_0$, with blocks $\tilde{Q}=0$, $\tilde{R}$, and $\tilde{S}$, and let $J'(t)$
    be a one-parameter family such that $Q'(t) = Q$, $R'(t) = R + t\tilde{R}$, and $S'(t) = S + t\tilde{S} + t^2 \tilde{R}^T Q^{-1} \tilde{R}$.
    By the Schur complement condition, we see that $\mrank\,J'(t) = d$ for all $t$.
    Also, $g_i(J'(t)) = 4-\mbk_i^T J'(t) \mbk_i = \langle \nabla g_i,\tilde{J} \rangle - t^2 \mbl^T \tilde{R}^T Q^{-1} \tilde{R} \mbl$
    is nonpositive for $t>0$ by the positive definiteness of $Q^{-1}$. Therefore, for some $\epsilon$, $J'(t)\in\mcR_0(4)$ for all $0\le t<\epsilon$.
    Since $\mathrm{obj}(J'(t)) = \mathrm{obj}(J)$ and $J$ is a local minimum of $\mathrm{obj}$ over $\mcR_0(4)$, so must $J'(t)$ for all sufficiently small $t$. Therefore $J$ is a fluid packing.
\end{proof}

We restricted the statement of Theorem~\ref{thm:deg} to the case $m=2$, where its analysis is more tractable thanks to the simplified form of the rank constraint (see Section~\ref{sec:rankconstraint}). However, we conjecture that it is also true for $m>2$, but the analysis of optimality conditions becomes more complicated. Theorems~\ref{thm:suff} and~\ref{thm:deg} suggest a straightforward generalization of Voronoi's algorithm to 2-periodic sets, which we elaborate in the next section.

\section{The generalized Voronoi algorithm}

\subsection{Outline}

Our algorithm seeks to enumerate all the locally optimal 2-periodic sets in $d$ dimensions. As we proved in Theorem \ref{thm:deg}, those are either fluid packings or algebraically extreme packings. The problem of enumerating the fluid packings is rather straightforward for $m=2$, since the two component lattices must themselves be extreme.

%Therefore, for each extreme $d$-dimensional lattice, each hole (circumcenter of Delone cell) that is deep enough to accommodate another translate of the lattice without reducing the packing radius can give rise to a fluid packing, and these different fluid packings may or may not be connected to each other via their flow. This informal description \alexei{and the basic method of detecting the fluid packings discussed below} is the extent to which we will discuss this part of the problem in this paper, and we will devote the remainder to the enumeration of the algebraically extreme lattices. In fact, in the dimensions where we have implemented our algorithm and present the results of the full enumeration, there are no fluid packings \alexei{(none were discovered by running our algorithm)}.

The enumeration of the algebraically extreme 2-periodic sets of packing radius $1$ consists of two steps that can be conceptually thought of as occurring one after the other, but in our implementation are actually interleaved. The first step is to enumerate all the vertices and edges of $\mcR(4)$ up to equivalence under the action of $\Gamma$. The second step is to take each edge, represented in the form $\{J+tJ':0\le t\le 1\}$ or $\{J+tJ':t\ge 0\}$, solve for all $t$ such $J+tJ'\in\mcR_0(4)$, and check whether $J+tJ'$ is algebraically extreme. We elaborate on these steps below.

\subsection{Enumerating connected components of fluid packings}
\label{sec:fluid}

As noted in Section~\ref{sec:density}, a 2-periodic fluid packing is the union of two translates of an extreme lattice. The relative translation must be a vector that is at a distance from the nearest lattice point of at least twice the packing radius. The local maxima of the distance from the nearest lattice point are called the \textit{holes} of a lattice, and are well-studied objects in the theory of lattices~\cite{conway1999sphere}. Namely, they are the circumcenters of the Delaunay cells of the lattice (when those lie inside the Delaunay cell). Every fluid packing is in the same connected component as the fluid packing obtained by taking the union of an extreme lattice and a translation of it by a hole of that lattice. The extreme lattices of a given dimension can be enumerated using Voronoi's algorithm, and their holes can be determined, for example, by finding the vertices of the Voronoi cell of the lattice~\cite{conway1999sphere}. This enumeration might be redundant if two holes of different lattice orbits are connected by a path whose distance to the lattice does not go below twice the packing radius, but it is guaranteed to give us at least one representative of every connected component of fluid packings.

\subsection{The shortest vector problem and related problems}

Throughout the algorithm, we need to solve a problem analogous to the problem known as the shortest vector problem (SVP) in the context of lattices~\cite{micciancio2010deterministic}. In our context, given a quadratic form $J:\mbR^{d+m-1}\to\mbR$ with positive definite upper-left $d\times d$ block $Q$, we wish to find its minimum over nonzero vectors in $\mathcal M = \mbZ^d\times(E-E)$. We also, in some cases, want to enumerated the vectors attaining this minimum, or more generally enumerate all vectors that attain a value below some threshold.

For $\mbk = (\mbn,\mbl)\in\mathcal M$ with fixed $\mbl$, $J(\mbk)$ is an inhomogeneous quadratic form of $\mbn$ (see Eq.~\eqref{eq:inhom}). Finding the minimum of an inhomogeneous quadratic form over the integer vectors is a problem known as the closest vector problem (CVP) for lattices and is closely related to the SVP~\cite{micciancio2010deterministic}. Since there are only finitely many possible values of $\mbl$, the problem reduces to a finite number of instances of the CVP problem for the underlying lattice~\cite{schurmann2009computational}.

As in the case of lattices, it might be expected that reduction of the quadratic form (using elements of $\Gamma$) would significantly decrease the average time to compute the SVP. To reduce the form, we first reduce the underlying lattice (the $Q$ block) by applying an operation $T$ of the form~\eqref{eq:Tblocks}, where $V=0$ and $W=1$. We then perform size reduction of the translation vectors, by applying the appropriate operation $T$, where $U=1$ and $W=1$, so that $Q^{-1}R^T$ has entries in $[-\tfrac12,\tfrac12]$. 

\subsection{Enumeration of vertices}

The algorithm to enumerate the vertices of $\mcR(4)$ is similar to the one accomplishing the analogous task in Voronoi's algorithm. We start with a known vertex of $\mcR(4)$, denoted $J_1$. We compute the extreme rays of its cone $C_{J_1}$~\cite{avis2017lrs}. For each such ray, $\{J_1+tJ':t\ge0\}$ there are two possibilities: either it lies entirely in $\mcR(4)$ or there is some $t>0$ such that $J_1+tJ'$ is another vertex of $\mcR(4)$, which we say is \textit{contiguous} to $J_1$. The first possibility does not exist in the original algorithm for lattices, but does occur in our case, and we must check for this possibility. We discuss this problem in Section~\ref{sec:rays}. If it is determined that the edge is not unbounded, we need to determine $t^* = \max\lbrace t: J_1+t J'\in\mcR(4)\rbrace$. This can be done by using bisection to narrow down the possible range of values for $t^*$ the point where we can identify the inequalities that become violated for $t > t^*$. We can then solve for $t^*$ by setting them to equalities. This would be a variation of the Algorithm $2$, $3.1.5$ from Ref.~\citep{schurmann2009computational}.
As $t^*$ is known to be rational, we find that in the low dimensional cases we tackle in this work, a practical and efficient method is to check whether $J_1+(a/b)J'$ is a vertex of $\mcR(4)$ for increasing natural values of $a$ and $b$. However, as the number of dimensions increases we expect this method to stop being practical, as it does in the case of the lattice Voronoi theory.

For each contiguous vertex, we check if it is equivalent to $J_1$ (more on this step in Section~\ref{sec:equiv}), and if not, we add it to a queue of vertices to be processed and to our partial enumeration of vertices. At each subsequent step of the algorithm, we remove a vertex from the queue, compute its contiguous vertices, check them for equivalence against all vertices in our partial enumeration, and add the ones that are not equivalent to previously enumerated vertices to the queue and to the partial enumeration. The enumeration is complete when the queue is empty.

As we mentioned in Section \ref{sec:symmetries}, we conjecture that $\mcR(4)$ has finitely many inequivalent vertices for all $d$, and therefore that the enumeration always halts. However, even if there are infinitely many vertex orbits, and the enumeration is performed breadth-first, then any orbit will eventually be enumerated by the algorithm.

The vertices of $\mcR(4)$ and the facets of their cones are all rational since they are determined by equations of the form $\mathbf{k}^T J\mathbf{k} = 4$ with integer coefficients. Therefore, all these calculations can be performed using exact arithmetic.

One way to obtain a starting vertex to initialize the algorithm is as follows: consider $A_d$ (or $D_d$, or any extreme, nonfluid $d$-dimensional lattice). It can be represented as 2-periodic set by taking a sublattice of index 2 instead of the primitive lattice. This 2-periodic set is necessarily algebraically extreme~\cite{schurmann2013strict}, and so it lies on an edge of $\mcR(4)$, which must terminate at a vertex in at least one direction.

\subsection{Detection of unbounded edges}
\label{sec:rays}

Given a vertex $J$ of $\mcR(4)$ and an extreme ray of $C_J$, of the form $\{J+t J':t\ge0\}$, we want to determine if the ray lies entirely in $\mcR(4)$. This is the case if and only if $J'(\mbk)\ge0$ for all $\mbk\in\mathcal M$. So, we have a problem very similar to the SVP above, except that we may not assume that $Q'$ is positive definite. If $Q'$ is not even positive semidefinite, then there is some $\mbk=(\mbn,0)$ such that $J'(\mbk) = Q'(\mbn) < 0$, and the ray is not unbounded. So the only remaining case is when $Q'$ is positive semidefinite and has nontrivial null space, $N(Q')$.

We break this remaining case into two cases. First, consider the case that there exist $\mbl\in(E-E)$ and $\mathbf{u}\in N(Q')$ with $c = \mathbf{u}^T (R')^T \mbl > 0$. Let $\mathbf{v}(t) = (-t c \mathbf{u}, \mbl)$, let $[\mathbf{v}](t)\in\mathcal M$ be the closest integer vector to $\mathbf{v}(t)$, and let $(\mathbf{e},0) = \mathbf{v}(t)-[\mathbf{v}](t)$ be the remainder. Then
\begin{equation}
J'([\mathbf{v}](t)) = t^2 c^2 \mathbf{u}^T Q' \mathbf{u} - t c \mathbf{u}^T Q' \mathbf{e}^T + \mathbf{e} Q' \mathbf{e} - 2t c^2 + \mbl^T S \mbl\text.
\end{equation}
All terms except $-2tc^2$ are either zero or bounded as $t\to\infty$ and so $J'([\mathbf{v}](t))<0$ for large enough $t$ and the ray is not unbounded.

The final case is that $\mathbf{u}^T (R')^T \mbl = 0$ for all $\mbl\in(E-E)$ and $\mathbf{u}\in N(Q')$. In this case, the problem reduces to the first case, where $Q'$ is positive definite, albeit in a smaller dimension: Since $Q'$ is a rational matrix, there is a unimodular transformation $U\in GL_d(\mbZ)$ such that $Q'\circ U$ has $\mathrm{span}\,\{\mathbf{e}_{r+1},\ldots,\mathbf{e}_d\}$ as its null space and is positive definite on $\mathrm{span}\,\{\mathbf{e}_{1},\ldots,\mathbf{e}_r\}$. We can find this transformation by computing the Hermite normal form or the Smith normal form of $Q'$ scaled to an integer matrix. Let $\mbn = \mbn' +\mbn''$ with $\mbn''\in N(Q'\circ U)$ and $\mbn'\in N(Q'\circ U)^\perp$. Then
\begin{equation}
J'(U(\mbn),\mbl) = (\mbn')^T U^T Q' U\mbn + 2 (\mbn')^T U^T (R')^T \mbl + 2 (\mbn'')^T U^T (R')^T \mbl + \mbl^T S' \mbl\text.
\end{equation}
Since, by assumption of this case, $(\mbn'')^T U^T (R')^T \mbl = 0$, the value of $J'(U(\mbn),\mbl)$ depends only on $\mbn'$ and $\mbl$, and we may find its minimum using our SVP method.

\subsection{Equivalence checking}
\label{sec:equiv}

Given two quadratic forms $J_1$ and $J_2$, we wish to determine if $J_2 = J_1\circ T = T^T J_1 T$ for some $T\in\Gamma$. Let us denote this relation as $J_1\sim J_2$. This problem may apply to vertices of $\mcR(4)$, as part of the algorithm for enumerating vertices, but we may also apply it to any pair of quadratic forms in $\mcR(4)$ that are not necessarily vertices. Let us first prove some useful results.

\begin{dfn}
    A set $M\subset\mathcal M$ is \textit{perfect} if $J(\mbk) = J'(\mbk)$ for all $\mbk\in M$ implies $J = J'$.
\end{dfn}

A set $M$ is perfect if and only if the set $\{\mbk\mbk^T: \mbk\in M\}$ spans the space of symmetric $(d+m-1)\times(d+m-1)$ matrices. Denote by $J^{-1}(A)\subset \mcM$ the set of vectors in $\mcM$ attaining values in the set $A$. If $A=\{4\}$ and $J$ is a vertex of $\mcR(4)$, then $J^{-1}(A)$ is perfect. A direct consequence of the definition of a perfect set is the following lemma:

\begin{lem}
    Let $A$ be a set of real values, and let $M_1 = J_1^{-1}(A)$ and $M_2 = J_2^{-1}(A)$ be the set of
    vectors in $\mathcal M$ that achieve these values. If $M_1$ and $M_2$ are perfect, then the following are equivalent:
    \begin{enumerate}
	\item $J_2 = J_1\circ T$ for some $T\in\Gamma$.
	\item $T(M_2) = M_1$, and $J_2|_{M_2} = (J_1\circ T)|_{M_2}$ for some $T\in\Gamma$.
    \end{enumerate}
    In particular a $T$ that satisfies one condition also satisfies the other.
\end{lem}
\begin{proof}
     First, suppose $J_2 = J_1\circ T$, and let $\mbk\in M_2$. Since $J_2(\mbk) = J_1(T\mbk)\in A$, we have that
     $T\mbk\in M_1$. Similarly, if $\mbk\in M_1$, then $T^{-1}\mbk\in M_2$. So $T(M_2) = M_1$, and clearly 
     $J_2|_{M_2} = (J_1\circ T)_{M_2}$ follows \textit{a fortiori} from the unrestricted equality.
 
     The other direction follows immediately from the definition of a perfect set.
\end{proof}

Therefore, a simple algorithm to check for equivalence is as follows: first, construct a set $A$ such that $M_1 = J_1^{-1}(A)$ is perfect. When $J_1$ is a vertex of $\mcR(4)$, the set $A=\{4\}$ suffices. Otherwise, we find the smallest $a>4$ such that $A=[4,a]$ suffices. Next, compute $M_2 = J_2^{-1}(A)$. If $M_2$ is not perfect, $|M_1|\neq|M_2|$, or $J_2$ does not take values in $A$ over $M_2$ with the same frequency as $J_1$ does over $M_1$, then $J_1\not\sim J_2$. We give labels $1,\ldots,s$ to the elements of $M_1=\{\mbk_1,\ldots,\mbk_s\}$ and $M_2=\{\mbk'_1,\ldots,\mbk'_s\}$, such that $\mbk_1,\ldots,\mbk_{d+m-1}$ are linearly independent (there must be such a linearly independent subset for $M_1$ to be perfect). We now try to construct an injective map $\sigma:\{1,\ldots,d+m-1\}\to\{1,\ldots,s\}$ such that
\begin{equation}
    \label{eq:permut-equiv}
    (\mbk'_{\sigma(i)})^T J_2\mbk'_{\sigma(j)} = \mbk_i^T T^T J_1 T \mbk_j\text{ for all }1\le i,j\le d+m-1\text.
\end{equation}
We can do this by a backtracking search, constructing $\sigma$ on $1,\ldots,n<d+m-1$ for increasing $n$, and backtracking when no possible assignment of $\sigma(n+1)$ satisfies \eqref{eq:permut-equiv} for $1\le i,j\le n+1$. For each such complete map produced by the backtracking search, there is a unique linear map $T'$ such that $T'\mbk_i = \mbk'_{\sigma(i)}$ for $1\le i\le d+m-1$. If $T'\in\Gamma$ and $T'(M_1) = M_2$, we are done and $J_1\sim J_2$. Otherwise, we continue with the backtracking search. If the backtracking search concludes without finding any equivalence, then $J_1\not\sim J_2$.

\subsection{Enumeration of algebraically extreme forms}

Given an edge of $\mcR(4)$ of the form $J+tJ'$, where $0\le t\le1$ or $t\ge0$, we wish to identify the points of this edge that lie in $\mcR_0(4)$. As we are limiting ourselves to the case $m=2$, we simply need to solve $\det(J+tJ') = 0$ for $t$, which is a univariate polynomial equation. This equation can be solved for $t$ as a generalized eigenvalue problem $J\mathbf{x} = -t J^\prime\mathbf{x}$. It is possible that the entire edge lies in $\mcR_0(4)$, as happens for one edge in our enumeration for $d=5$. In that case, either $Q$ is constant over the edge, and any locally optimal packing on the edge is necessarily a fluid packing, or $Q$ is not constant, and therefore neither is $\mathrm{obj}$. So, only points that minimize $\mathrm{obj}$ over the edge --- necessarily endpoints by quasiconcavity --- can be locally optimal packings.

For each possible candidate configuration (each root of $\det(J+t J')$ if it is not identically zero, or the endpoints if it is), we test to see if it is algebraically extreme. If $J$ is algebraically extreme, we can certify that it is by considering the dual problem: $J$ is algebraically extreme if and only if the cone $\{\sum_{\mbk\in J^{-1}(4)} \eta_{\mbk} \mbk\mbk^T + \mathbf{x}\mathbf{x}^T : \eta_\mbk>0, \mathbf{x}\in N(J)\}$ is full-dimensional and includes $G_J$. If $J$ is not algebraically extreme, we can certify that by the direct problem of finding $J'$ that gives a counterexample for the definition: $J'\neq J$, $\langle J'-J,G_J\rangle\le0$, and $J'\in C_J\cap T_J$. Since $t$ comes from the root of a univariate polynomial, it is not, in general, rational.

In our calculation, we used floating point representation for the candidate algebraically extreme configurations and for the implementation of the test for whether they are algebraically extreme.

\section{Numerical results}
\label{sec:nr}

We use an implementation of the algorithm described in the previous section to fully enumerate the vertices of the Ryshkov-like polyhedron $\mcR(4)$ and the algebraically extreme $m$-periodic packings for $m=2$ and $d=3,4,5$. The results of the enumeration are summarized in Table \ref{tab:summary}. The extreme lattices in these dimensions are the classical root lattices $A_d$, $D_d$, and $E_5$, whose covering radii are known~\cite{conway1999sphere}, and do not exceed twice the packing radius. Since the existence of a 2-periodic fluid packing requires the existence of an extreme lattice with a hole that is at a distance of more than twice the packing radius from the nearest lattice point (see Section \ref{sec:fluid}), there are no fluid packings for $d=3,4,5$. Our enumeration of the algebraically extreme 2-periodic algebraically extreme packings exhausts all locally optimal 2-periodic packings.

Our attempted enumeration for $d=6$ did not terminate after over a month of running. The main bottleneck appears to be the enumeration of extreme rays of $C_J$ for vertices $J$ with large number of minimal vectors $J^{-1}(4)$ (see Sec.~\ref{sec:nr:d6}). This is similar to the main bottleneck in the enumeration of perfect lattices in $d=8$, where only by exploiting the symmetries of these high-kissing-number lattices, was the full enumeration made tractable~\cite{sikiric2007classification}. We are hopeful that a similar approach could be used for $m=2$ to make full enumeration in higher dimensions than $d=5$ tractable, but we do not attempt to implement it in this work.

We use a mix of floating-point arithmetic and exact arithmetic in various parts of the algorithm. Therefore, we do not claim the results in this section as rigorous results, but rather as numerical results. We relied on floating-point arithmetic for the shortest vector problem, giving sufficient tolerance so that no candidate shortest vectors are rejected and then using exact arithmetic to verify that they are in fact all of the same length. Determining the rays of vertex cones and finding contiguous vertices along those rays was done with exact rational arithmetic. The final step of finding candidate optima along the edges of the polyhedron and checking whether these candidates are algebraically extreme was done using floating-point arithmetic. After obtaining the set of algebraic extreme packings, we were able to use a computer algebra system to obtain exact algebraic expressions for their coordinates.

\begin{table}
    \begin{tabular}{lccc}
	vertices of $\mcR(4)$ & 4 (2) & 10 (6) & 34 (25)\\
	algebraically extreme 2-periodic sets & 3 (1) & 7 (3) & 29 (20) \\
	highest density & $1/(2\sqrt{2})$ & $1/8$ & $1/(8\sqrt{2})$ \\
	multiplicity of highest density & 3 (1) & 2 (0) & 5 (2)
    \end{tabular}
    \caption{Summary of enumeration results. For $d = 3, 4, 5$ we list the number of vertices of the Ryshkov-like polyhedron, the number of algebraically extreme 2-periodic sets, the highest number density achieved for packing radius $1$, and the number of algebraically extreme sets that achieve this density. In parentheses, we indicate how many of the corresponding 2-periodic sets are not also lattices.}
    \label{tab:summary}
\end{table}

\subsection{d=3}

\begin{figure}
    \includegraphics[width=0.32\textwidth]{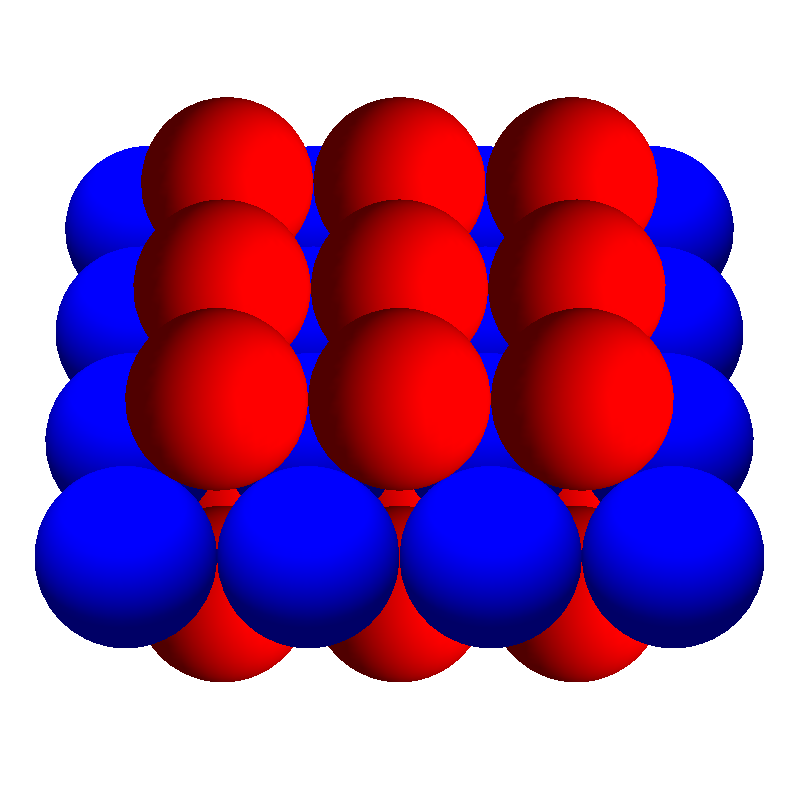}
    \includegraphics[width=0.32\textwidth]{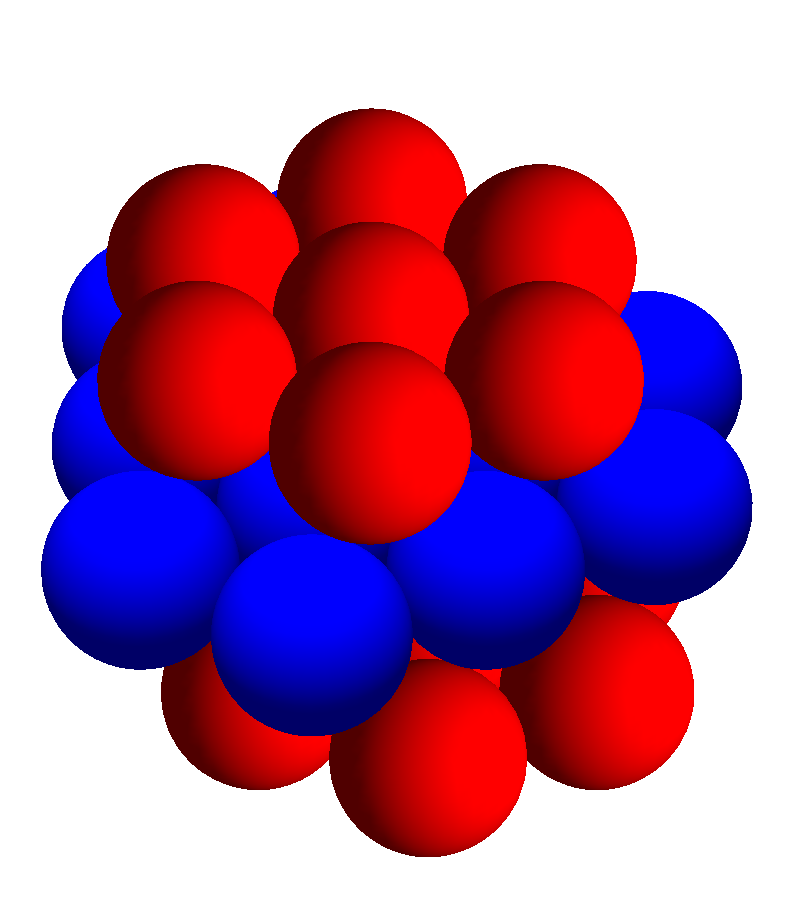}
    \includegraphics[width=0.32\textwidth]{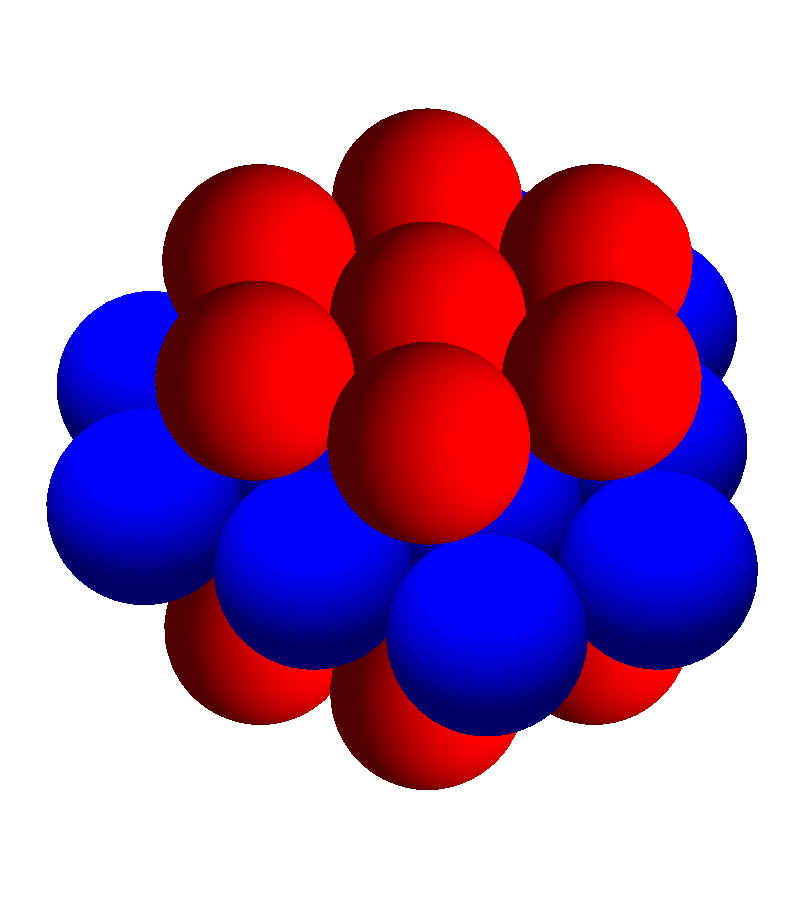}
    \caption{The three algebraically extreme 2-periodic packings in 3 dimensions. We use red and blue spheres to represent the two orbits under translation by lattice vectors. Both the left and middle packings are representations of the fcc lattice as a 2-periodic set, but they are not equivalent when the two orbits are distinguished. The right packing is the hexagonal close-packed 2-periodic arrangement and is the only algebraically extreme 2-periodic packing in 3 dimensions that is not also a lattice.}
    \label{fig:3dpacks}
\end{figure}

The enumeration in $d=3$ gave 4 inequivalent vertices and 3 inequivalent algebraically extreme 2-periodic sets. All the algebraically extreme 2-periodic sets in 3 dimensions of packing radius $1$ have the same density $\delta = 1/(2\sqrt{2})\approx0.354$, and they are represented by the following quadratic forms:
\begin{equation}
    J_{3,1} = 2\left(\begin{array}{cccc}2 & 0 & 0 & -1\\ 0 & 2 & 0 & 1\\ 0 & 0 & 4 & 2\\ -1 & 1 & 2 & 2\end{array}\right)\quad
    J_{3,2} = 2\left(\begin{array}{cccc}2 & -1 & 1 & 0\\ -1 & 2 & 0 & 1\\ 1 & 0 & 6 & 3\\ 0 & 1 & 3 & 2\end{array}\right)
\end{equation}
\begin{equation}
    J_{3,3} = \frac{2}{3}\left(\begin{array}{cccc}6 & 3 & 0 & 0\\ 3 & 6 & 0 & 3\\ 0 & 0 & 16 & 8\\ 0 & 3 & 8 & 6\end{array}\right)
\end{equation}
The forms $J_{3,1}$ and $J_{3,2}$ are two inequivalent representations of the fcc lattice as a 2-periodic set. $J_{3,1}$ is a stacking of square layers, whereas $J_{3,2}$ is a stacking of triangular layers (see Figure~\ref{fig:3dpacks}). The forms are not equivalent under the action of $\Gamma$ because the corresponding two sublattices of the fcc lattices are not equivalent under the symmetries of the fcc lattice. The form $J_{3,3}$ represents the hexagonal-close-packing 2-periodic arrangement, which is not a lattice. Note that a form $J$ is the representation of a lattice as a 2-periodic set if and only if $Q^{-1}R^T\in(\tfrac12\mbZ)^d$.

All the algebraically extreme 2-periodic sets have the same kissing number, 12, but as quadratic forms they have different number of minimal vectors: $|J_{3,1}^{-1}(4)|=20$, $|J_{3,2}^{-1}(4)|=|J_{3,3}^{-1}(4)|=18$. This difference occurs because contacts between spheres in the same orbit contribute the same minimal vector as the analogous contact in a different orbit. Denote by $|\cdot|_*$ a counting measure that gives weight $m$ to vectors of the form $\mbk=(\mbn,0)$. Then the average kissing number is $\kappa = |J^{-1}(\mathrm{min}\, J)|_*/m$. 

\begin{figure}
    \includegraphics[width=0.24\textwidth]{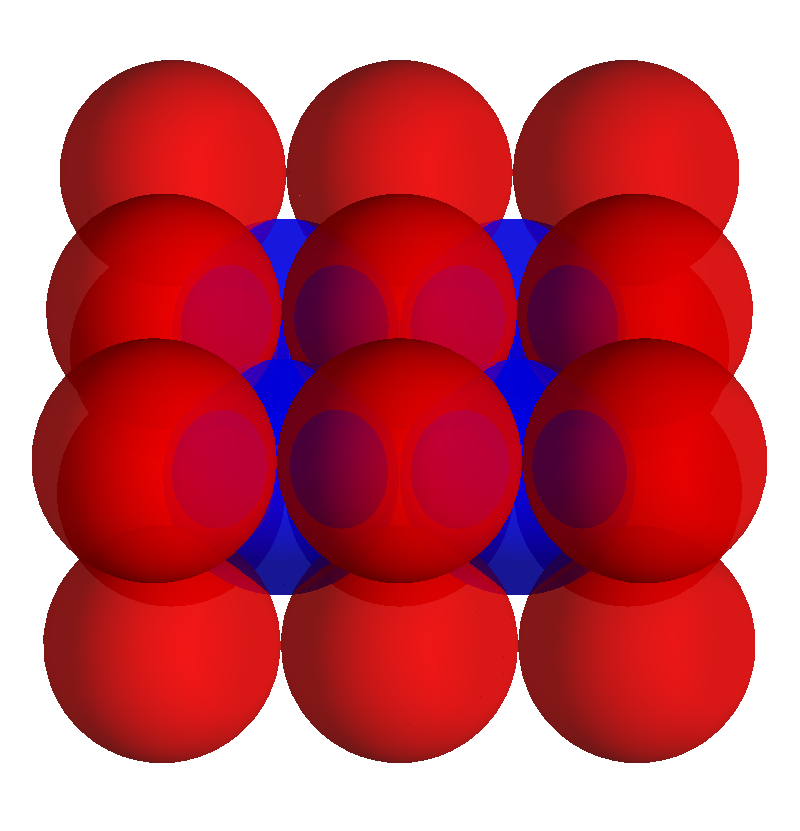}
    \includegraphics[width=0.24\textwidth]{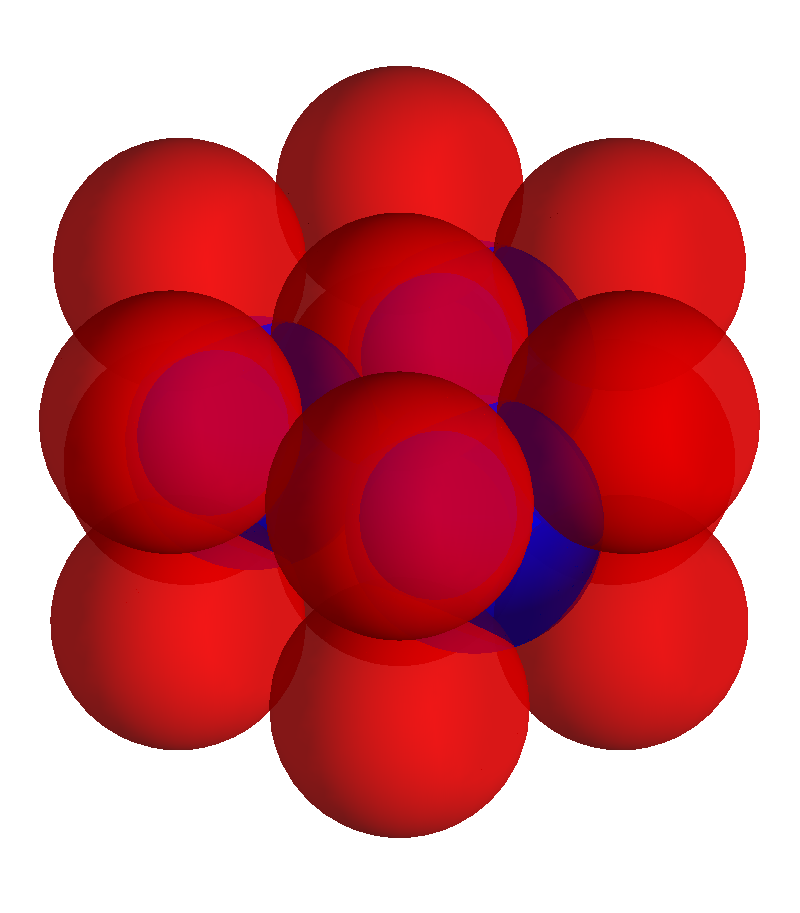}
    \includegraphics[width=0.24\textwidth]{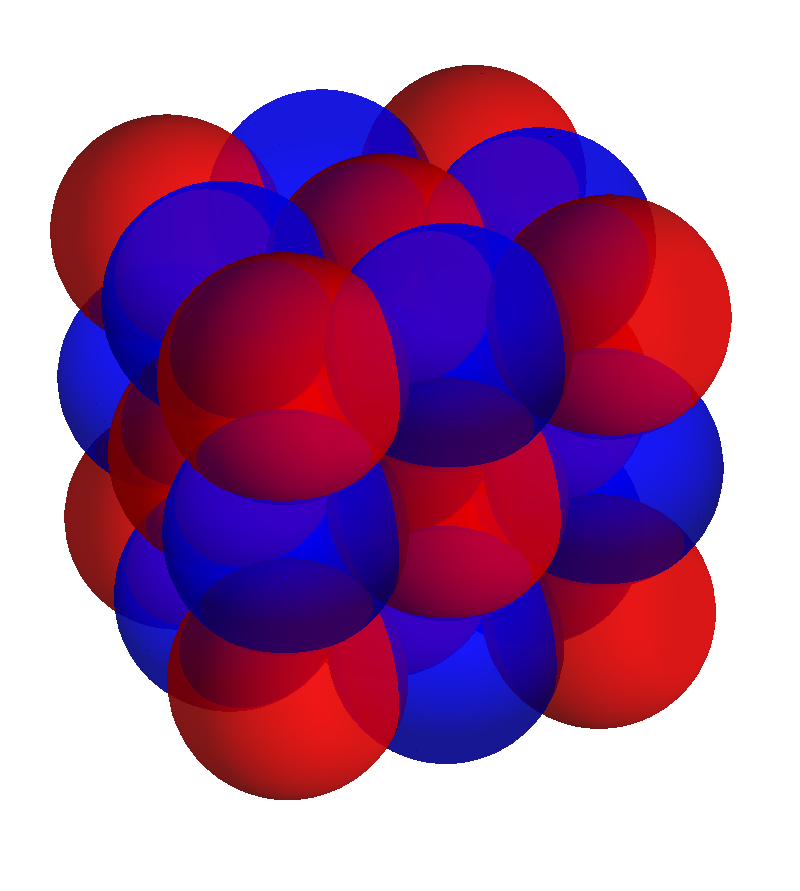}
    \includegraphics[width=0.24\textwidth]{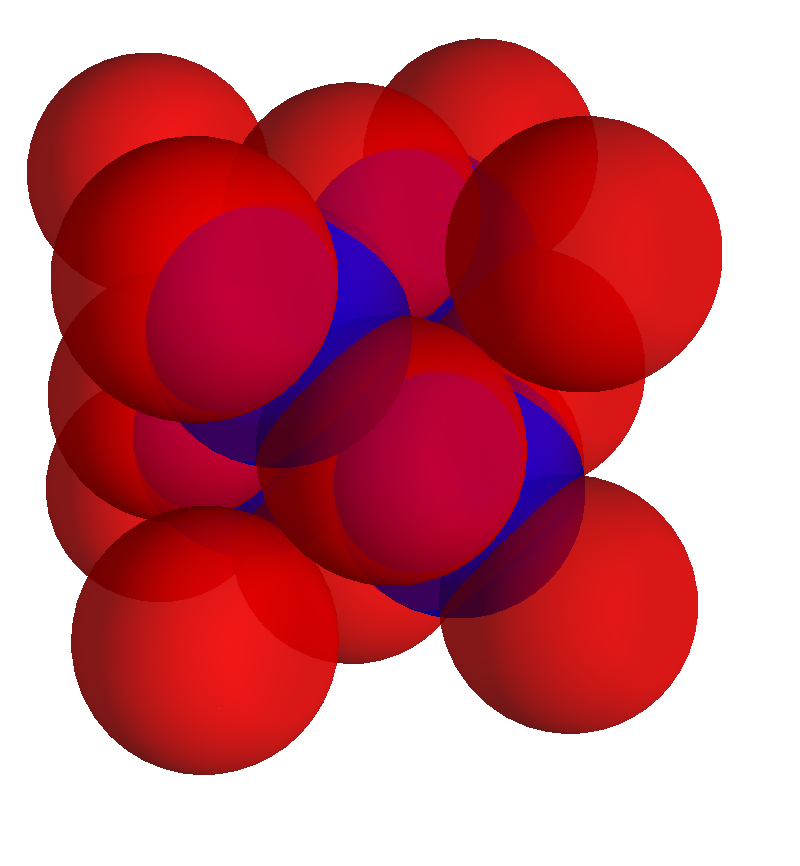}
    \caption{The vertices of the Ryshkov-like polyhedron for $d=3$ and $m=2$ can be interpreted as binary
    packings of nonadditive hard spheres.}
    \label{fig:nonadd}
\end{figure}

The vertices of the Ryshkov-like polyhedron are also interesting to consider. They are
\begin{equation}
    J_{3,1v} = 2\left(\begin{array}{cccc}2 & 0 & 0 & 1\\ 0 & 2 & 0 & -1\\ 0 & 0 & 2 & 1\\ 1 & -1 & 1 & 2\end{array}\right)\quad
    J_{3,2v} = 2\left(\begin{array}{cccc}2 & 1 & 0 & 0\\ 1 & 2 & 0 & 1\\ 0 & 0 & 2 & 1\\ 0 & 1 & 1 & 2\end{array}\right)
\end{equation}
\begin{equation}
    J_{3,3v} = 2\left(\begin{array}{cccc}2 & 1 & -1 & 0\\ 1 & 2 & 0 & 1\\ -1 & 0 & 2 & 1\\ 0 & 1 & 1 & 2\end{array}\right)\quad
    J_{3,4v} = 2\left(\begin{array}{cccc}2 & -1 & -1 & 0\\ -1 & 2 & 1 & 1\\ -1 & 1 & 2 & 0\\ 0 & 1 & 0 & 2\end{array}\right)\text.
\end{equation}
They have full rank, so they do not correspond to 2-periodic sets in 3 dimensions. However, because they are positive definite, they correspond to 4-dimensional sets that are periodic (with two orbits) under a 3-dimensional lattice. When they are projected to the space spanned by the lattice, they can be interpreted as binary packings of non-additive spheres, where the two sphere species have equal self-radius, but smaller radius when interacting with each other. The first is a simple cubic lattice of one species with its body-center holes filled with spheres of the other species. The second is a hexagonal lattice with one of the two inequivalent triangular prism-shaped holes in each unit cell filled. The third and fourth are the fcc lattice with its octahedral or (one of its) tetrahedral holes filled, giving, respectively, a simple cubic lattice with alternating species (the NaCl crystal structure) and the diamond crystal structure (see Figure~\ref{fig:nonadd}).

\begin{figure}
    \centering
    \includegraphics[scale=0.6]{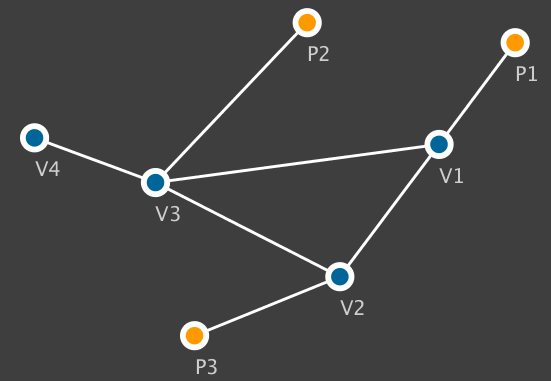}
    \caption{The Voronoi graph of 2-periodic sets in $d=3$, a generalization of the Voronoi graph for lattices. 	The blue nodes are the vertices of $\mcR(4)$, the Ryshkov-like polyhedron; the orange nodes are points of $\mcR_0(4)$ lying on unbounded edges (rays) of $\mcR(4)$. Two vertices are connected in the graph if they are contiguous, and edge points are connected to the vertices on which their edge is incident. The labels $Vn$ and $Pn$ corresponds to $J_{3,nv}$ and $J_{3,n}$ respectively. In $d=3$ all the points of $\mcR_0(4)$ lying on polyhedron edges are on unbounded edges and they are all algebraically extreme.}
    \label{fig:d3-graph}
\end{figure}

Finally, we point out as an example of an unbounded edge, the edge connecting $J_{3,3}$ and $J_{3,2v}$.

\subsection{d=4}

In 4 dimensions, we obtain 7 inequivalent algebraically extreme 2-periodic sets. Six of those lie in the relative interior of edges of the Ryshkov-like polyhedron, and one is a vertex. Two algebraically extreme sets achieve the maximum density, $\delta = 1/8 = 0.125$, and both are representations of the lattice $D_4$ as a 2-periodic set. They are represented by the forms,
\begin{equation}
    J_{4,1} = 2\left(\begin{array}{ccccc}2 & 0 & 0 & 0 & 1\\ 0 & 2 & 0 & 0 & -1\\ 0 & 0 & 2 & 0 & -1\\ 0 & 0 & 0 & 2 & -1\\ 1 & -1 & -1 & -1 & 2\end{array}\right)\quad
    J_{4,2} = 2\left(\begin{array}{ccccc}2 & -1 & -1 & 0 & -1\\ -1 & 2 & 1 & 0 & 1\\ -1 & 1 & 2 & 0 & 0\\ 0 & 0 & 0 & 4 & 2\\ -1 & 1 & 0 & 2 & 2\end{array}\right)\text.
\end{equation}
The form $J_{4,1}$ is a vertex of the Ryshkov-like polyhedron. Both have the kissing number of $D_4$, $\kappa=24$, as the average kissing number. The next-highest density achieved is $\delta=2/\sqrt{144+64\sqrt{5}}\approx0.118$, which is achieved by a single algebraically extreme 2-periodic set:
\begin{equation}
    J_{4,3} = 2\left(\begin{array}{ccccc}2 & -1 & -1 & -1 & -1\\ -1 & 2 & 1 & 1 & 0\\ -1 & 1 & 2 & 1 & 1\\ -1 & 1 & 1 & 2+2\tau & -\tau\\ -1 & 0 & 1 & -\tau & 2\end{array}\right)\text,
\end{equation}
where $\tau = (1+\sqrt{5})/2$ is the golden ratio. Its average kissing number is $\kappa=22$. Finally, there are four algebraically extreme 2-periodic sets that achieve the same density as the $A_4$ lattice, $\delta=1/(4\sqrt{5})\approx0.112$. Two of them are in fact representations of the $A_4$ lattice, but the other two are not also lattices. All four have the same kissing number as the lattice, $\kappa=20$. Including $J_{4,1}$, there are 10 inequivalent vertices of the Ryshkov-like polyhedron.

\begin{figure}
    \centering
    \includegraphics[scale=0.6]{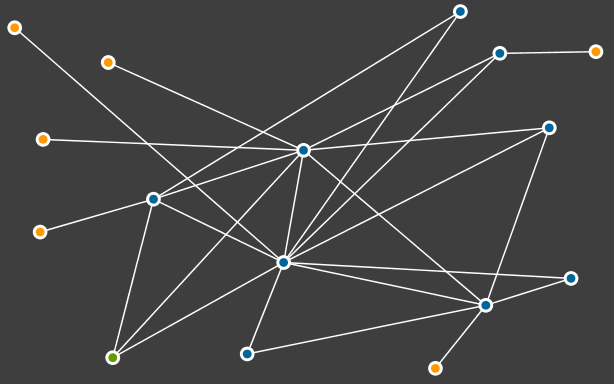}
    \caption{The Voronoi graph in $d=4$. The blue nodes are polyhedron vertices that are not in $\mcR_0(4)$. The green and orange nodes are points of $\mcR_0(4)$ lying on vertices and unbounded edges of the polyhedron, respectively. In $d=4$, as in $d=3$, all the points of $\mcR_0(4)$ in the relative interior of polyhedron edges are on unbounded edges and are all algebraically extreme. We omit the labels for neatness.}
    \label{fig:d4-graph}
\end{figure}

\subsection{d=5}

In 5 dimensions, we obtain 29 inequivalent algebraically extreme 2-periodic sets (see Table \ref{tab:d5ext} for a tabulation of their invariants). Five inequivalent 2-periodic sets achieve the maximum density $\delta = 1/(8\sqrt{2})\approx0.0884$, three of which are representations of the lattice $D_5$ as a 2-periodic set, and two are nonlattices that achieve the same density. The ones that represent the lattice $D_5$ are given by the forms,
\begin{equation}
    J_{5,1} = 2\left(\begin{array}{cccccc}2 & -1 & -1 & 0 & 0 & 1\\ -1 & 2 & 1 & 0 & 0 & 0\\ -1 & 1 & 2 & 0 & 0 & -1\\ 0 & 0 & 0 & 2 & 0 & -1\\ 0 & 0 & 0 & 0 & 2 & -1\\ 1 & 0 & -1 & -1 & -1 & 2\end{array}\right)
\end{equation}
\begin{equation}
    J_{5,2} = 2\left(\begin{array}{cccccc}2 & -1 & -1 & -1 & 0 & -1\\ -1 & 2 & 1 & 1 & 0 & 1\\ -1 & 1 & 2 & 0 & 0 & 0\\ -1 & 1 & 0 & 2 & 0 & 1\\ 0 & 0 & 0 & 0 & 4 & 2\\ -1 & 1 & 0 & 1 & 2 & 2\end{array}\right)
\end{equation}
\begin{equation}
    J_{5,3} = 2\left(\begin{array}{cccccc}2 & -1 & -1 & -1 & -1 & -1\\ -1 & 2 & 1 & 1 & 1 & 1\\ -1 & 1 & 2 & 1 & 1 & 0\\ -1 & 1 & 1 & 2 & 1 & 1\\ -1 & 1 & 1 & 1 & 4 & -1\\ -1 & 1 & 0 & 1 & -1 & 2\end{array}\right)\text,
\end{equation}
where $J_{5,1}$ is also a vertex of the Ryshkov-like polyhedron (the only vertex that is also in $\mcR_0$). The two that are not lattices but achieve the same density are
\begin{equation}
    J_{5,4} = \left(\begin{array}{cccccc}4 & -2 & -2 & 2 & -2 & -2\\ -2 & 4 & 0 & -2 & 2 & 0\\ -2 & 0 & 4 & -2 & 2 & 2\\ 2 & -2 & -2 & 4 & -2 & -2\\ -2 & 2 & 2 & -2 & 10 & 5\\ -2 & 0 & 2 & -2 & 5 & 4\end{array}\right)
\end{equation}
\begin{equation}
    J_{5,5} = \frac{2}{5}\left(\begin{array}{cccccc}10 & 5 & 5 & -5 & 0 & 0\\ 5 & 10 & 5 & -5 & 0 & 5\\ 5 & 5 & 10 & -5 & 0 & 5\\ -5 & -5 & -5 & 10 & 0 & -5\\ 0 & 0 & 0 & 0 & 16 & -8\\ 0 & 5 & 5 & -5 & -8 & 10\end{array}\right)\text.
\end{equation}
All have the same kissing number $\kappa=40$.

The next highest density, $\delta = (\sqrt{5 - 2 \sqrt{6}})/4 \approx 0.0795$, is achieved by three algebraically extreme 2-periodic sets, all of which are not also lattices:
\begin{equation}
    J_{5,6} = \frac{2}{5}\left(\begin{array}{cccccc}10 & -5 & -5 & 5 & -5 & -5\\ -5 & 10 & 5 & -5 & 5 & 0\\ -5 & 5 & 10 & -5 & 0 & 0\\ 5 & -5 & -5 & 10 & -5 & -5\\ -5 & 5 & 0 & -5 & 16 + 4  \sqrt{6}  & 8 + 2  \sqrt{6} \\ -5 & 0 & 0 & -5 & 8 + 2  \sqrt{6}  & 10\end{array}\right)\\
\end{equation}
\begin{equation}
    J_{5,7} = \frac{2}{5}\left(\begin{array}{cccccc}10 & 5 & -5 & -5 & -5 & 10\\ 5 & 10 & -5 & -5 & 0 & 10\\ -5 & -5 & 10 & 5 & 0 & -5\\ -5 & -5 & 5 & 10 & 0 & -5\\ -5 & 0 & 0 & 0 & 14 + 4  \sqrt{6}  & 2 + 2  \sqrt{6} \\ 10 & 10 & -5 & -5 & 2 + 2  \sqrt{6}  & 20\end{array}\right)\\
\end{equation}
\begin{equation}
    J_{5,8} = \frac{1}{2}\left(\begin{array}{cccccc}8 & 4 & 4 & 0 & -4 & 0\\ 4 & 8 & 4 & 0 & -4 & 4\\ 4 & 4 & 8 & 0 & -4 & 4\\ 0 & 0 & 0 & 8 & 0 & -4\\ -4 & -4 & -4 & 0 &  8 + 2  \sqrt{6}  & -4 -  \sqrt{6} \\ 0 & 4 & 4 & -4 & -4 - \sqrt{6}  & 8\end{array}\right)\text.
\end{equation}
The first two of these have a kissing number of $\kappa=35$, and the third has $\kappa = 34$. The third largest density is that achieved by one of the three extreme lattices in 5 dimensions, $\delta = 1/(9\sqrt{2})\approx0.0786$, and is achieved by five 2-periodic sets, three of which are representations of the lattice, and all have the same kissing number as the lattice. The third extreme lattice, $A_5$, has the lowest density of all extreme lattices, $\delta = 1/(8\sqrt{3})\approx0.722$, and this density is achieved by seven 2-periodic sets, three of which represent the lattice, and all having the same kissing number as the lattice. This is also the lowest density achieve by any algebraically extreme 2-periodic set, although there is a form of $\mcR_0(4)$ on an edge of $\mcR(4)$ that has lower density, but fails to be algebraically extreme.

Including $J_{5,1}$, the Ryshkov-like polyhedron has 34 vertices. Two of the vertices are not positive semidefinite, and this is the lowest dimension where such vertices occur. When interpreted as nonadditive binary sphere packings, these packings would have a nonself-radius larger than the self-radius. In this sense they are similar to the fluid packings (the distance between orbits is larger than the distance within each orbit), but the underlying lattices are not extreme. Another interesting phenomenon that first occurs in 5 dimensions is an edge of $\mcR(4)$ that is completely contained in $\mcR_0(4)$. Such an edge is the one connecting $J_{5,1}$ to $J_{5,1}\circ T$, where
\begin{equation}
    T = \left(\begin{array}{cccccc}1 & 0 & -1 & 0 & -1 & 1\\ 0 & 1 & 0 & 0 & 0 & 0\\ 0 & 0 & 0 & 0 & -1 & 1\\ 0 & 0 & 0 & 1 & 0 & 0\\ 0 & 0 & -1 & 0 & 0 & 1\\ 0 & 0 & 0 & 0 & 0 & 1\end{array}\right)\text,
\end{equation}
as well as all the equivalent edges. Along the edge $J_{5,1}+t(J_{5,1}\circ T-J_{5,1})$, the objective determinant is $\mathrm{obj}(t)=512(1 + t - t^2)$, so any internal point cannot be a locally optimal 2-periodic set.

\begin{table}
    \scalebox{0.9}{
    \begin{tabular}{ccccc}
	\# & $\delta$ & $\delta$ & L/N & $\kappa$\\
	1 & 0.0884 & $1/(8 \sqrt{2})$ & L & 40 \\
2 & 0.0884 & $1/(8 \sqrt{2})$ & L & 40 \\
3 & 0.0884 & $1/(8 \sqrt{2})$ & L & 40 \\
4 & 0.0884 & $1/(8 \sqrt{2})$ & N & 40 \\
5 & 0.0884 & $1/(8 \sqrt{2})$ & N & 40 \\
6 & 0.0795 & $\sqrt{5 - 2 \sqrt{6}}/4$ & N & 35 \\
7 & 0.0795 & $\sqrt{5 - 2 \sqrt{6}}/4$ & N & 35 \\
8 & 0.0795 & $\sqrt{5 - 2 \sqrt{6}}/4$ & N & 34 \\
9 & 0.0786 & $1/(9 \sqrt{2})$ & L & 30 \\
10 & 0.0786 & $1/(9 \sqrt{2})$ & L & 30 \\
11 & 0.0786 & $1/(9 \sqrt{2})$ & L & 30 \\
12 & 0.0786 & $1/(9 \sqrt{2})$ & N & 30 \\
13 & 0.0786 & $1/(9 \sqrt{2})$ & N & 30 \\
14 & 0.0771 & $\sqrt{419 + 1011 \sqrt{33}}/1024$ & N & 30 \\
15 & 0.0765 & $-\sqrt{2} + (2 \sqrt{5})/3$ & N & 29 \\
16 & 0.0765 & $-\sqrt{2} + (2 \sqrt{5})/3$ & N & 29 \\
17 & 0.0758 & $3/(2 \sqrt{126 + 4 \sqrt{19} \cos {\alpha_1} - 266 \cos {2 \alpha_1}})$ & N & 30 \\
18 & 0.0750 & $3/40$ & N & 30 \\
19 & 0.0748 & $1/(36 \sqrt{-58 + 26 \sqrt{5}})$ & N & 26 \\
20 & 0.0748 & $(3 \sqrt{3})/((\sqrt{(3 - 4 \cos {\alpha_2} - 8 \cos {2 \alpha_2})})(4 (5 + 4 \cos {\alpha_2})))$ & N & 28 \\
21 & 0.0738 & $59049/(4 \sqrt{x_3})$ & N & 25 \\
22 & 0.0737 & $5/(48 \sqrt{2})$ & N & 30 \\
23 & 0.0722 & $1/(8 \sqrt{3})$ & L & 30 \\
24 & 0.0722 & $1/(8 \sqrt{3})$ & L & 30 \\
25 & 0.0722 & $1/(8 \sqrt{3})$ & L & 30 \\
26 & 0.0722 & $1/(8 \sqrt{3})$ & N & 30 \\
27 & 0.0722 & $1/(8 \sqrt{3})$ & N & 30 \\
28 & 0.0722 & $1/(8 \sqrt{3})$ & N & 30 \\
29 & 0.0722 & $1/(8 \sqrt{3})$ & N & 30

    \end{tabular}}
    \caption{Some invariants of the 29 algebraically extreme 2-periodic sets in $d=5$ dimensions, including density and kissing number. The \textit{L/N} column indicates whether this is the representation of a lattice as a 2-periodic set (\textit{L}) or a truly nonlattice arrangement (\textit{N}). Here 
	$\alpha_1 = \tfrac13(\tan^{-1}(3\sqrt{762})-\pi)$,
	$\alpha_2 = \tfrac13(\tan^{-1}(9\sqrt{47}/17)-\pi)$,
	$\alpha_3 = \tfrac13(\tan^{-1}(486\sqrt{1077}/5867)-\pi)$, and
$x_3=8189832078 + 22432320 \sqrt{661} \cos {\alpha_3} - 35353733064 \cos {2 \alpha_3} - 3050582422 \cos {4 \alpha_3} + 6990736 \sqrt{661} \cos {5\alpha_3}$.}
    \label{tab:d5ext}
\end{table}

\begin{figure}
    \centering
    \includegraphics[scale=0.5]{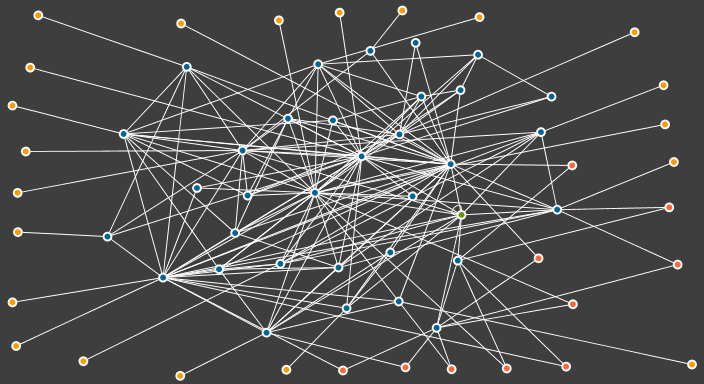}
    \caption{The Voronoi graph in $d=5$. The notation is the same as in Figs.~\ref{fig:d3-graph} and~\ref{fig:d4-graph}. This is the lowest dimension where a point of $\mcR_0(4)$ lying in the relative interior of an edge connecting two vertices appears (red nodes). This is also the lowest dimension where an entire edge of the polyhedron is in $\mcR_0(4)$, but since only endpoints of such edges can be locally optimal, we do not include such edge points in the graph. We omit the node labels for neatness.}
    \label{fig:d5-graph}
\end{figure}

\subsection{d=6}
\label{sec:nr:d6}

In $d=6$ we were only able to perform partial enumeration starting from a vertex incident on the edge on which a representation of $A_6$ lies. The partial enumeration discovered $730$ vertices of $\mcR(4)$ and $692$ points of $\mcR_0(4)$ on edges of $\mcR(4)$ (including 8 vertices). There are also $4$ edges connecting $3$ vertices lying completely in $\mcR_0(4)$. The obstacles that stalled the full enumeration are three vertices with particularly complex cones. One is a representation of $E_6$ as a 2-periodic set (with $|J^{-1}(4)|=124$) and two vertices with $|J^{-1}(4)|=112$ and $|J^{-1}(4)|=126$ minimal vectors:
\begin{gather}
    J_{6,E_6} = 2\left(
        \begin{array}{ccccccc}
            6 & -2 & 0 & 0 & 0 & 0 & 3\\
            -2 & 2 & -1 & 0 & 0 & 0 & -1\\
            0 & -1 & 2 & -1 & 0 & -1 & 0\\
            0 & 0 & -1 & 2 & -1 & 0 & 0\\
            0 & 0 & 0 & -1 & 2 & 0 & 0\\
            0 & 0 & -1 & 0 & 0 & 2 & 0\\
            3 & -1 & 0 & 0 & 0 & 0 & 2\\
        \end{array}
    \right),\\
    J_{6,112} = 2\left(
        \begin{array}{ccccccc}
            2 & -1 & 1 & -1 & -1 & -1 & 0\\
            -1 & 2 & 0 & 0 & 0 & 0 & 1\\
            1 & 0 & 2 & -1 & -1 & -1 & 1\\
             -1 & 0 & -1 & 2 & 1 & 1 & -1\\
             -1 & 0 & -1 & 1 & 2 & 1 & 0\\
             -1 & 0 & -1 & 1 & 1 & 2 & 0\\
             0 & 1 & 1 & -1 & 0 & 0 & 2\\
        \end{array}
    \right),\\
    J_{6,126} = 2\left(
        \begin{array}{ccccccc}
            2 & -1 & -1 & -1 & -1 & 1 & -1\\
            -1 & 2 & 1 & 1 & 1 & -1 & 0\\
            -1 & 1 & 2 & 0 & 0 & -1 & 1\\
            -1 & 1 & 0 & 2 & 1 & 0 & 0\\
            -1 & 1 & 0 & 1 & 2 & -1 & 0\\
            1 & -1 & -1 & 0 & -1 & 2 & -1\\
            -1 & 0 & 1 & 0 & 0 & -1 & 2\\
        \end{array}
    \right).
\end{gather}
The number of faces of the cone is half the number of minimal vectors. For comparison, the $E_7$ lattice (as a 1-periodic set) has $126$ minimal vectors, and the enumeration of the extreme rays of its cone is already barely tractable by brute force. Since these forms are highly symmetric (have large automorphism groups), the number of symmetry orbits of the extreme rays is much smaller than the total number of extreme rays. By using a method that exploits the symmetry of these forms, the calculation would hopefully become tractable. This strategy was successful, for example, in making the enumeration of perfect lattice in $d=8$ tractable \cite{sikiric2007classification}.

\section{Conclusion}

Using our generalization of the Voronoi algorithm to 2-periodic sets, we were able to produce a complete enumeration of the locally optimal 2-periodic sphere packings in dimensions $d=3,4,$ and $5$. In particular, we show that it is impossible to obtain a higher density using 2-periodic arrangements in these dimensions than is possible with lattices. However, in $d=3$ and $d=5$ (but not in $d=4$), there are nonlattice 2-periodic arrangement that match the optimal lattice packing density.

Our work leaves a number of important open questions, whose solution will enable application of this method to higher values of $m$ and $d$:
\begin{enumerate}
    \item We were not able to prove \textit{a priori} that the Ryshkov-like polyhedron must have a finite number of faces (in particular vertices) up to the action of $\Gamma$. For $m=2$ and $d=3,4,5$, this follows directly from the fact that our enumeration halts. However, it would be good to know that the enumeration is always guaranteed to halt for any $m$ and $d$.
    \item Theorem \ref{thm:deg} is proven only for $m=2$, but we conjecture that it holds also for $m>2$. If this conjecture holds, our algorithm can be immediately extended to $m>2$. However, it would now involve looking for points of $\mathcal {R}_0(4)$ that lie on $\tfrac12m(m+1)$-dimensional faces of $\mcR(4)$. This would significantly increase the complexity of two steps the algorithm. First, instead of enumerating extreme rays of $C_J$ for vertices $J$, we need to enumerate the $\tfrac12m(m+1)$-dimensional faces of $C_J$. Second, instead of solving for $S-RQ^{-1}R^T=0$ over a line, we need to solve it over a $\tfrac12m(m+1)$-dimensional space.
    \item To make the enumeration for $m=2$ and $d>5$ tractable, it would be helpful to make use of the symmetries of highly symmetric vertices. For any vertex $J$, we need only enumerate the orbits of the extreme rays of $C_J$ under the automorphism group of $J$, $\mathrm{Aut}(J)\subset \Gamma$. Such methods were used to make the enumeration of perfect lattices tractable in $d=8$ \cite{sikiric2007classification}.
    \item In dimensions where a full enumeration might not be tractable, heuristic optimization methods could be useful for discovering new packing arrangements as well as providing empirical backing to conjectures about certain arrangements being optimal. Again, in the case of lattices ($m=1$), such methods have been remarkably successful, reproducing the densest known lattices in up to $d=20$ dimensions. Stochastic enumeration, traversing the 1-skeleton of the Ryshkov polyhedron by picking a random contiguous vertex at each step~\cite{andreanov2012random,andreanov2016extreme}, can be applied to the Ryshkov-like polyhedron treated here. Sequential linear programming methods~\cite{marcotte2013efficient,kallus2013jammed} and simulated annealing methods~\cite{kallus2013statistical} can also be used to sample periodic arrangements. Our results and those of Sch\"urmann~\cite{schurmann2010perfect,schurmann2013strict} can be used to certify locally optimal packings found.
    \item The theory presented here for packings of equal-sized spheres can be straightforwardly generalized to mixtures of differently-sized spheres in numerical proportions that are whole multiples of $1/m$. The constraints $J(\mathbf{n},\mathbf{l})\le \lambda$ need to be replaced by ones where the right-hand-side is a function of the species of the two spheres involved, $\lambda(\mathbf{l})$. This generalization can handle the case of additive or non-additive spheres.
\end{enumerate}

%\begin{acknowledgments}
	\textit{Acknowledgments} A.\ A.\ was supported by Project Code (IBS-R024-D1). Y.\ K.\ was supported by an Omidyar Fellowship at the Santa Fe Institute.
%\end{acknowledgments}

\appendix

\section{Locally finite polyhedra}\label{sec:lfp}

\begin{dfn}\label{def:lfp}
    If the intersection of every compact polytope with the closed convex set $P\subset\mathbb{R}^n$ is a compact
    polytope, then $P$ is a \textit{locally finite polyhedron}.
    If, additionally, no line is a subset of $P$, then $P$ is \textit{locally finite polyhedron with no lines}.
    The \textit{dimension} of $P$ is the dimension of its affine span.
\end{dfn}

The empty set is considered a locally finite polyhedron of dimension $-1$. A locally finite polyhedron of dimension $0$ is a point. A locally finite polyhedron of dimension $1$ is a line segment, a ray, or a line.

%\begin{prp}\label{prop:lfp-intersection}
%    The intersection of two locally finite polyhedra is a locally finite polyhedron (with no lines).
%\end{prp}

%\begin{proof}
%    Since the intersection of a compact polytope with a locally finite polyhedron is a compact polytope,
%    the proposition follows immediately.
%\end{proof}

\begin{dfn}\label{def:face}
    Let $P$ be a locally finite polyhedron and let $f: \mathrm{aff} P \to \mathbb{R}$ be a
    linear function such that $f(x) \le 0$ for every $x\in P$.
    Then $F=\{x\in P: f(x) = 0\}$ is a \textit{face} of $P$, and we say that $f$ \textit{supports} $F$.
\end{dfn}

Since the intersection of a compact polytope with an affine subspace is a compact polytope, the faces of a locally finite polyhedron are also locally finite polyhedra. The $k$-dimensional faces of an $n$-dimensional locally finite polyhedron are called vertices, and edges for $k=0$ and $1$ respectively.

\begin{prp}\label{prop:face-aff}
    If $F$ is a face of $P$, then $F = \mathrm{aff} F \cap P$.
\end{prp}
\begin{proof}
    Because $F$ is the intersection of an affine subspace of $\mathrm{aff} P$ with $P$,
    and any affine subspace containing $F$ contains $\mathrm{aff} F$,
    we have $\mathrm{aff} F \cap P \subseteq F$.
    From $F\subseteq P$, we have $F\subseteq \mathrm{aff} F \cap P$.
\end{proof}

\begin{prp}\label{prop:local-face}
    Let $P\subseteq\mathbb{R}^n$ be a locally finite polyhedron
    and let $K\subset\mathbb{R}^n$ be an $n$-dimensional compact polytope.
    If $F$ is a $d$-dimensional face of $P$ supported by $f$ and the interior of $K$ intersects $F$,
    then $F\cap K$ is a $d$-dimensional face of $P\cap K$ supported by $f$.
    Conversely, if $G$ is a $d$-dimensional face of $P\cap K$ supported by $g$ and the interior of $K$ intersects $G$,
    then $\mathrm{aff} G\cap P$ is a $d$-dimensional face of $P$ supported by $g$,
\end{prp}

\begin{proof}
    We have that $F\cap K = \{x\in P: f(x) = 0\}\cap K = \{x\in P\cap K: f(x) = 0\}$,
    and that $P\cap K\subseteq P \subseteq \{x:f(x)\le 0\}$.
    Since the affine span of a convex set is the same as the affine span of any
    neighborhood of a point in such a set, we also get that $F\cap K$ is $d$-dimensional.

    To prove the second part of the proposition, let $x$ be some point in the intersection of $G$
    with the interior of $K$.
    First, if there is some $y\in P$ with $g(y)>0$, then $x' = (1-\lambda)x+\lambda y \in P$ has $g(x')>0$.
    Since $x'\in K$ for $\lambda$ small enough, this is a contradiction, so $g(x)\le 0$ for all $x\in P$.

    Similarly, if there is some $y\in P\setminus \mathrm{aff} G$ with $g(y)=0$,
    then $x' = (1-\lambda)x+\lambda y$ has $g(x') = 0$.
    Since $x'\in K$ for $\lambda$ small enough, this is also a contradiction,
    so $\{y\in P: g(y) = 0\} \subset \mathrm{aff} G$. 
    Since $g(y) = 0$ for all $y\in\mathrm{aff} G$, we now have that $\mathrm{aff} G\cap P$ is a face of $P$
    supported by $g$.
    Since $G\subseteq\mathrm{aff} G\cap P\subseteq\mathrm{aff} G$, its dimension is the same as that of $G$.
\end{proof}

\begin{prp}\label{prop:face-of-face}
    A face of a face of $P$ is a face of $P$.
    Conversely, a face of $P$ contained in a face $F$ of $P$ is a face of $F$.
\end{prp}

\begin{proof}
    Let $F$ be a face of $P$ and let $G$ be a face of $F$.
    Let $T=[-t, t]^n$ be a box large enough so its interior intersects $G$.
    Then, by Proposition \ref{prop:local-face}, $F\cap T$ is a face of $P\cap T$ and $G\cap T$ is a face of $F\cap T$.
    Since each of these is a compact polytope, $G\cap T$ is a face of $P\cap T$.
    By Proposition \ref{prop:local-face}, $\mathrm{aff} G\cap P$ is a face of $P$.
    Finally
    $\mathrm{aff} G\cap P 
    = \mathrm{aff} G \cap \mathrm{aff} F \cap P
    = \mathrm{aff} G \cap F = G$, where the two last equalities follow from Proposition \ref{prop:face-aff}.

    For the converse, let $G$ be a face of $P$ that is supported by $g$ and contained in a face $F$.
    Then it immediately follows from the definition of a face,
    that the restriction of $g$ to the affine span of $F$ supports $G$ as a face of $F$.
\end{proof}

Note that the first clause of the proposition does not hold in general for convex sets, even bounded ones. Consider, for example, a stadium shape (the Minkowski sum of a disk and a line segment): the line segments on the boundary are 1-dimensional faces, but their endpoints are not 0-dimensional faces of the set.

The set of faces of $P$ partially ordered by inclusion is called the \textit{face complex} of $P$. The union of the faces of $P$ of dimension up to $k$ is called the $k$-skeleton of $P$.

\begin{prp}\label{prop:v-cone}
    Let $v$ be a vertex of a $d$-dimensional locally finite polyhedron $P$, then the intersection
    of all the closed halfspaces that contain $P$ and have $v$ on their boundary 
    is a $d$-dimensional polyhedral cone with no lines.
\end{prp}

\begin{proof}
    Let $T = v + [-1,1]^n$. Then $v$ is a vertex of $P\cap T$.
    The closed halfspaces that contain $P$ and have $v$ on their boundary are exactly
    the closed halfspaces that contain $P\cap T$ and have $v$ on their boundary.
    Since $P\cap T$ is a $d$-dimensional compact polytope, the proposition follows.
\end{proof}

We call this the vertex cone at $v$. The vertex cone can also be constructed as the conic hull of the edges that are incident on $v$.

\begin{prp}\label{prop:1skel-connected}
    The 1-skeleton of a locally finite polyhedron with no lines is connected.
\end{prp}

\begin{proof}
    Let $S_1$ and $S_2$ be distinct connected components of the 1-skeleton of $P$.
    Since $P$ has no lines, $S_1$ has a vertex of $P$.
    Therefore, let $v\in S_1$ a vertex of $P$ supported by $f$.
    Let $m = \sup_{S_2} f < 0$.
    Either $f(x) = m$ for some $x\in S_2$, or $f(x_n)\to m$ for a sequence
    of points $x_n\in S_2$ with distance from $v$ diverging to infinity.

    In the former case, there is a vertex $v'\in S_2$ with $f(v') = m$.
    Since its vertex cone as a vertex of $P$ is the conic hull of the edges incident on it,
    then $f(x) \le m$ for all $x$ in the vertex cone.
    Since $P$ is contained in the vertex cone, this is a contradiction, because $f(v)=0>m$.

    In the latter case, let $y_n$ be the intersection of the segments connecting $v$ to $x_n$ and
    the sphere of radius 1 around $v$. Since $y_n\in P$ and $f(y_n)\to 0$, the sequence accumulates
    at a point $y_\infty\in P$ with $f(y_\infty)=0$ and $\|y_\infty - v\| = 1$.
    This is also a contradiction, because $\{x\in P: f(x) = 0\} = \{v\}$.
\end{proof}

\bibliography{packing}

\end{document}